\documentclass{amsart}
\usepackage{amsmath,amssymb}
\usepackage{graphicx,subfigure,psfrag}
\usepackage[latin1]{inputenc}
\newtheorem{theorem}{Theorem}[section]
\newtheorem{proposition}[theorem]{Proposition}
\newtheorem{corollary}[theorem]{Corollary}

\newtheorem{question}{Question}
\newtheorem{lemma}[theorem]{Lemma}

\theoremstyle{definition}
\newtheorem{definition}[theorem]{Definition}
\newtheorem{example}[theorem]{Example}

\theoremstyle{remark}
\newtheorem{remark}[theorem]{Remark}

\numberwithin{equation}{section}

\newcommand{\al}{\alpha}
\newcommand{\be}{\beta}
\newcommand{\de}{\delta}
\newcommand{\ep}{\epsilon}

\newcommand{\ga}{\gamma}

\newcommand{\la}{\lambda}
\newcommand{\om}{\omega}
\newcommand{\si}{\sigma}

\newcommand{\vp}{\varphi}

\newcommand{\De}{\Delta}
\newcommand{\Ga}{\Gamma}
\newcommand{\La}{\Lambda}
\newcommand{\Si}{\Sigma}
\newcommand{\Om}{\Omega}
\newcommand{\tS}{\tilde{S}}
\newcommand{\tL}{\tilde{L}}

\newcommand{\tv}{\tilde{v}}

\newcommand{\hPhi}{\widehat{\Phi}}
\newcommand{\hPsi}{\widehat{\Psi}}

\def\BB{\mathbb{B}}

\def\RR{\mathbb{R}}

\def\ZZ{\mathbb{Z}}

\def\TT{\mathbb{T}}

\renewcommand\SS{\mathbb{S}}
\newcommand{\cA}{{\mathcal A}}

\newcommand{\cK}{{\mathcal K}}

\newcommand{\cM}{{\mathcal M}}

\newcommand{\cP}{{\mathcal P}}

\newcommand{\cU}{{\mathcal U}}

\newcommand{\pd}{\partial}
\newcommand\minus\backslash
\newcommand{\id}{{\rm{id}}}

\newcommand\lan\langle
\newcommand\ran\rangle

\newcommand{\supp}{\operatorname{supp}}

\newcommand{\I}{{\mathrm i}}
\newcommand{\e}{{\mathrm e}}

\DeclareMathOperator\diag{diag} 
\DeclareMathOperator\dist{dist}

\newcommand\DD{\mathbb{D}}

\renewcommand\leq\leqslant
\renewcommand\geq\geqslant
\newlength{\intwidth}

\newcommand\loc{_{\mathrm{loc}}}
\newcommand\BOm{\overline\Om}

\def\hTe{\widehat\Theta}

\def\hw{\hat w}
\begin{document}

\title[Level sets of solutions to an elliptic PDE]{Submanifolds that are level sets of solutions to a second-order elliptic PDE}

\author{Alberto Enciso and Daniel Peralta-Salas}
\address{Instituto de Ciencias Matemáticas (CSIC-UAM-UC3M-UCM),
  Consejo Superior de Investigaciones Científicas, C/ Nicol\'as Cabrera 15, 28049 Madrid, Spain}
\email{aenciso@icmat.es, dperalta@icmat.es}


%
%
\begin{abstract}
  Motivated by a question of Rubel, we consider the problem of
  characterizing which noncompact hypersurfaces in $\RR^n$ can be
  regular level sets of a harmonic function modulo a $C^\infty$
  diffeomorphism, as well as certain generalizations to other PDEs. We
  prove a versatile sufficient condition that shows, in particular,
  that any nonsingular algebraic hypersurface whose connected components are all noncompact
  can be transformed onto a union of components of the zero set of a
  harmonic function via a diffeomorphism of $\RR^n$. The technique we
  use, which is a significant improvement of the basic strategy we
  recently applied to construct solutions to the Euler equation with
  knotted stream lines (Ann.\ of Math.\ 175 (2012) 345--367), combines robust but not explicit local constructions with appropriate global approximation theorems. In view of applications to a problem of Berry and Dennis, intersections of level sets are also studied.
\end{abstract}
\maketitle

\section{Introduction}

A long-standing open problem on level sets of solutions to elliptic PDEs, formulated by L.A.\ Rubel and included in~\cite[Problem 3.20]{CCH80} and~\cite[Problem R.7]{LN88}, is to characterize harmonic functions in $\RR^n$ up to homeomorphism. That is, given a continuous function $f$ on $\RR^n$, one would like to know whether there is a homeomorphism $\Phi$ of $\RR^n$ such that $u(x):=f(\Phi(x))$ is harmonic. In particular, a necessary condition for the existence of $\Phi$ is that the collection of all level sets of $f$ must be homeomorphic to that of a harmonic function.

Actually, Rubel's problem was motivated by the foundational results of Kaplan~\cite{Ka48} and Boothby~\cite{Bo51} in the theory of foliations, who studied a less stringent version of Rubel's problem in the case of dimension $n=2$. This relaxed formulation, which is more natural from a topological viewpoint and goes back to Morse~\cite{Ma46}, is tantamount to characterizing the collection of level sets of a harmonic function up to homeomorphism. Equivalently, given a continuous function $f$ on $\RR^n$, the question is whether there exist a harmonic function $u$ in $\RR^n$ and a homeomorphism $\Phi$ of $\RR^n$ which maps connected components of any level set $f^{-1}(c)$ into connected components of level sets of $u$.

Even with Kaplan and Boothby's relaxed formulation, Rubel's problem is wide open. The easiest case is that of dimension $n=2$, where complex analytic methods are of great help in the study of the singular foliations defined by harmonic functions~\cite{Ka48,Bo51,ZM08}. Particularly, these authors used Stoilow's theorem to show that, for any continuous function $f$ on $\RR^2$ whose level sets satisfy certain necessary conditions, there are a homeomorphism $\Phi$ mapping the unit disk $\DD^2$ onto $\RR^2$ and a harmonic function $v$ on $\DD^2$ such that $v(x)=f(\Phi(x))$. In regard to harmonic functions on the whole plane instead of on the disk, these authors derived nontrivial topological obstructions (both for the relaxed and Rubel's formulations of the problem) on the admissible families of level sets of $f$ as a consequence of Picard's theorem, and constructed infinitely many non-homeomorphic families of level sets of harmonic functions. A sufficient criterion for $f$ to be homeomorphic to a harmonic polynomial was obtained by Shiota~\cite{Sh81}.

To the best of our knowledge, there are no results related to Rubel's problem in any dimension $n\geq3$. The reason for this is twofold. Analytically, all existing results in the case $n=2$ have been obtained using complex analytic techniques, which are no longer applicable when $n\geq3$. Topologically, the situation is much simpler when $n=2$ because any connected component of a regular level set of a function $u$ on $\RR^2$ is homeomorphic to a line or to a circle, and the latter possibility must be obviously excluded in the case where $u$ is harmonic. On the contrary, in higher dimension there are infinitely many topological types of level sets and it is not obvious which ones can actually be level sets of a harmonic function.

Our objective in this paper is to develop some tools to analyze which hypersurfaces $L$ in
$\RR^n$ can be regular level sets of a harmonic function up to
diffeomorphism. In particular, we aim to explore to what extent there
is a large collection of smooth, properly embedded hypersurfaces that can be level sets of a
harmonic function. Roughly speaking, this would measure the complexity
of the (singular) foliations one would need to consider in any approach to the higher dimensional Rubel problem. Other than the obvious obstruction that the hypersurface $L$ must be noncompact and the analysis of explicit examples of harmonic functions, there are no results on the admissible topological types of $L$. For instance, in dimension $n=3$ it is easy to construct explicit harmonic functions having level sets with a connected component homeomorphic to elementary surfaces such as the plane or the cylinder, but nothing is known regarding more complicated objects as in the following

\begin{question}\label{Q.1}
  Is there a harmonic function $u$ in $\RR^3$ such that $u^{-1}(0)$ has a connected component homeomorphic to the genus $g$ torus with $N$ ends, to the torus of infinite genus  or to the infinite jungle gym? (cf.\ Figure~\ref{F.jungle}).
\end{question}

\begin{figure}[t]
  \centering
  \subfigure{
\includegraphics[scale=0.2,angle=0]{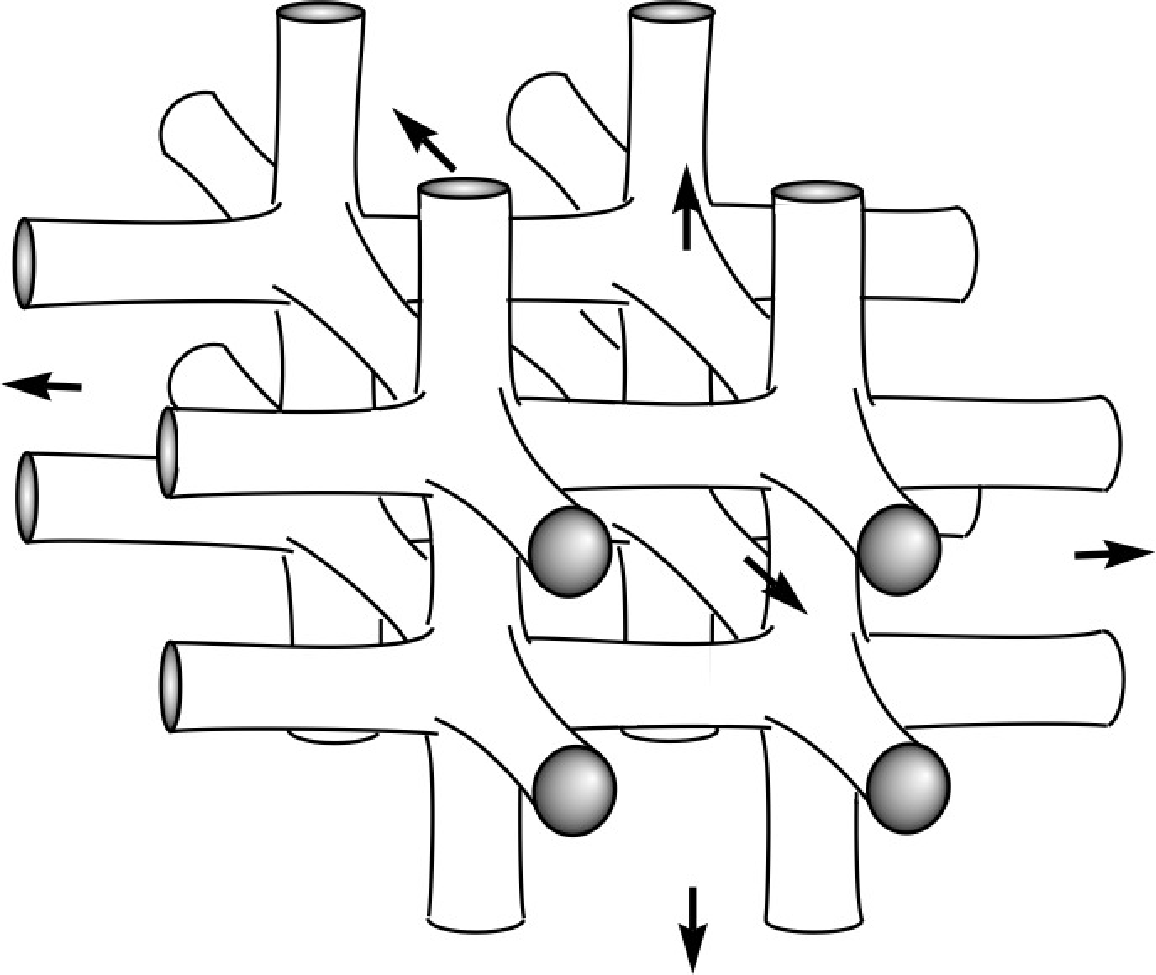}}\hspace{1.5em}
  \subfigure{
\includegraphics[scale=0.09,angle=69]{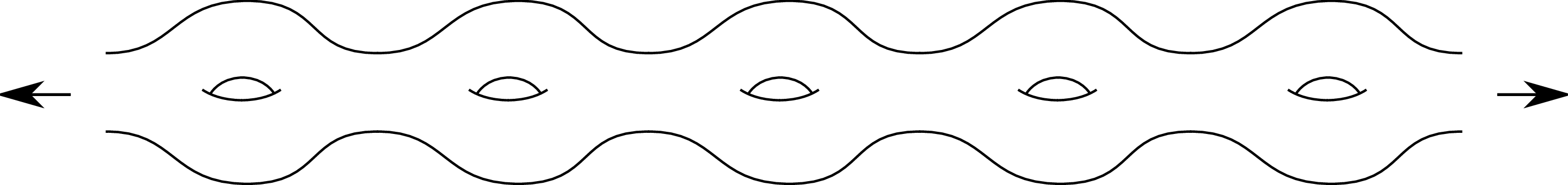}}\hspace{1.5em}
  \subfigure{
\includegraphics[scale=0.23,angle=0]{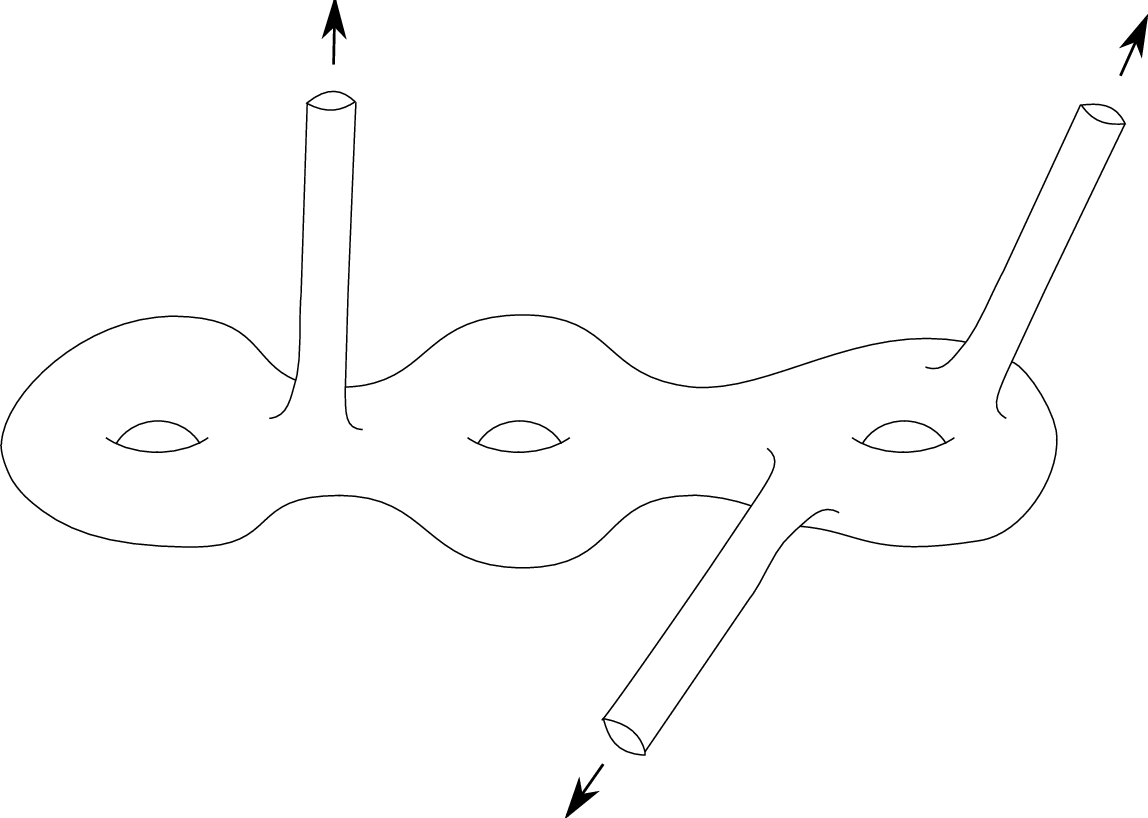}}
\caption{From the left to the right, an infinite jungle gym, a torus of infinite genus and a torus of genus $g=3$ and $N=3$ ends.} 
\label{F.jungle}
\end{figure}

Our main theorem provides a sufficient condition for a hypersurface to
be a connected component of the zero set of a harmonic function in
$\RR^n$ up to a diffeomorphism, thereby furnishing a positive answer
to Question~\ref{Q.1} as a particular case. More generally, we will
prove this result for the equation $(\De-k^2)u=0$, with $k$ a real
constant, and for possibly disconnected hypersurfaces. A somewhat
related result, which gives sufficient conditions for a set of curves
in the plane to be homeomorphic to the zero set of a harmonic
polynomial, has been recently established in~\cite{EJN07}.

Thinking for concreteness in the three-dimensional situation, the strategy of the proof is to construct a {\em local}\/ harmonic function having a level set diffeomorphic to the surface (for instance, the genus $g$ torus with $N$ ends, the torus of infinite genus  or the infinite jungle gym, as in Question~\ref{Q.1}) whose gradient and nearby level sets satisfy certain (rather subtle) geometric conditions. Once this function has been constructed, we show that there is a {\em global}\/ harmonic function that approximates the local solution uniformly in a neighborhood of the surface, and prove that this global solution has a level set diffeomorphic to the surface using a noncompact stability theorem and the geometric properties of the local solution. In order to carry out this strategy, it is crucial to deform the surface to ensure that its geometry is controlled at infinity, for example by transforming the collared ends of the surface into straight cylinders (as in the case of the torus with $N$ ends) or by making the surface invariant under a discrete group of isometries (as in the case of the infinite jungle gym or the torus of infinite genus). As a consequence of this, the homeomorphism that appears in Question~\ref{Q.1} will be obtained as the composition of a `large' diffeomorphism used to control the geometry at infinity of the surface and a `small' diffeomorphism related to the approximation of the local harmonic function. A simpler application of the philosophy of combining a robust local construction with a suitable global approximation theorem has been recently applied to study periodic trajectories of solutions to the Euler equation~\cite{AoM}.

Because of this two-stage construction of the diffeomorphism, the theorem can be most conveniently formulated in terms of a natural $n$-dimensional analog of these surfaces with controlled geometry at infinity, which we call {\em tentacled hypersurfaces}. The precise definition of the latter is given in Sections~\ref{S.finite} and~\ref{S.infinite}. For the purposes of this Introduction the reader can stick to the three-dimensional case as above, keeping in mind that, for instance, any hypersurface obtained from a compact manifold by removing patches and attaching `nicely embedded' collared ends (the `tentacles') is diffeomorphic to a tentacled hypersurface.

\begin{theorem}\label{T.1}
  Let $L$ be a tentacled hypersurface of $\RR^n$ (possibly disconnected
  and of infinite type) and let $k$ be a real constant. Then one can transform the hypersurface $L$ by a smooth diffeomorphism $\Phi$ of $\RR^n$, arbitrarily close to the identity in the $C^1$ norm, so that $\Phi(L)$ is a union of connected components of a level set $u^{-1}(c_0)$, where the function $u$ satisfies the equation $(\De-k^2)u=0$ in $\RR^n$.
\end{theorem}

It should be emphasized that the class of hypersurfaces that can be
realized as a tentacled hypersurface modulo diffeomorphism is
deceivingly wide; in particular, this includes all nonsingular real algebraic
hypersurfaces whose components are all noncompact, which allows us to state the following

\begin{corollary}\label{C.algebraic}
Let $L$ be a smooth embedded hypersurface of $\RR^n$ whose connected
components are all noncompact. Suppose $L$ is a nonsingular real
algebraic hypersurface. Then,
given any real constant $k$, the hypersurface $L$ can be transformed by a smooth diffeomorphism $\Phi$ of $\RR^n$, not necessarily small, so that $\Phi(L)$ is a union of connected components of a level set of a function $u$ satisfying the equation $(\De-k^2)u=0$ in $\RR^n$.
\end{corollary}

As a matter of fact, the proof of this result ensures that the
singular foliation defined by the level sets of the solution $u$ is a
trivial bundle in a neighborhood of $\Phi(L)$. That is, there is a
neighborhood of $\Phi(L)$ saturated by level sets of $u$ where the
foliation is equivalent to the product $\Phi(L)\times (-1,1)$.

A word of caution regarding the terminology is in order. When we say
that $L$ is a smooth, embedded hypersurface, we mean that all its
connected components are codimension-$1$, $C^\infty$, properly
embedded submanifolds of $\RR^n$. We recall that a nonsingular real
algebraic hypersurface can be written as the zero set of a polynomial
in $\RR^n$ whose gradient does not vanish at any point of the hypersurface. It should also be noticed that the
algebraic hypersurface $L$ considered in Corollary~\ref{C.algebraic}
does not have any singular points or self-intersections and, as such,
encloses no bounded domains. Without the smoothness requirement
(which is essentially used in our approach, since the proofs are based
on ``differentiable'' methods), $L$ would not need to be homeomorphic
to a union of connected components of a harmonic function, as shown by
the example $L:=\{(x_1^2+x_2^2+x_3^2-1)x_3=0\}$ in $\RR^3$.

The subtleties that appear in the proof of these results are related
to the necessity of combining the rigid, quantitative methods of
analysis with the flexible, qualitative techniques of differential
topology. A simple yet illustrative example is the {\em necessity}\/
of allowing for a small diffeomorphism $\Phi\neq\id$ in the above
theorems: indeed, it is well known that the curves $x_2=x_1^s$, which
are all diffeomorphic for any integer $s\geq1$, can be a connected
component of a level set of a harmonic function in $\RR^2$ if and only
if $s=1$ or $s=2$~\cite{FNS66}.

After discussing the level sets of a harmonic function, we will next consider regular joint level sets of $m$ harmonic functions, i.e., transverse intersections of their individual level sets. Unlike a single level set, these intersections can be compact provided the number of functions $m$ is at least $2$. Again, nothing is known on this problem, particularly if one is interested in sets diffeomorphic to pathological objects as in the following

\begin{question}\label{Q.2}
  Let $L$ be an exotic sphere of dimension $m$, so that $L$ is a smooth manifold homeomorphic to the standard $m$-sphere but not diffeomorphic to it. It is well known that $L$ can be smoothly embedded in $\RR^{2m}$. Are there $m$ harmonic functions $u_r$ in $\RR^{2m}$ such that the exotic sphere $L$ is diffeomorphic to a component of the joint level set $u_1^{-1}(0)\cap\cdots \cap u_m^{-1}(0)$?
\end{question}

Our main motivation to consider intersections of level sets comes from Berry and Dennis' question~\cite{BD01} of whether any finite link (that is, a collection of pairwise disjoint knots) can be a a union of connected components of the zero set of a complex-valued solution to the~PDE $(\De+k^2)u=0$ in $\RR^3$. Their question, which is of interest in the study of dislocation structures, is a particular case of the following

\begin{question}\label{Q.3}
  Let $L$ be a locally finite link in $\RR^3$. Are there two solutions $u_r$ of the equation $(\De+k^2)u_r=0$ in $\RR^3$ such that $L$ is diffeomorphic to a union of connected components of the joint level set $u_1^{-1}(0)\cap u_2^{-1}(0)$?
\end{question}

We shall next state our main results on joint level sets of solutions to an elliptic PDE, which in particular provide affirmative answers to Questions~\ref{Q.2} and~\ref{Q.3}. Here, our results and the method of proof depend greatly on whether the joint level set we want to prescribe is compact (or, more generally, a union of compact components) or not. The case of compact joint level sets is less involved and can be tackled for a very wide class of elliptic equations:

\begin{theorem}\label{T.3}
  Given an integer $m\geq2$, let $L$ be a locally finite union of
  pairwise disjoint, compact, codimension-$m$, embedded submanifolds of $\RR^n$ with trivial normal bundle. Take $m$ real constants $c_r$ and $m$ second-order elliptic differential operators $T_r$ with real analytic coefficients. Then one can transform the submanifold $L$ by a smooth diffeomorphism $\Phi$ of $\RR^n$, arbitrarily close to the identity in any $C^p$ norm, so that $\Phi(L)$ is a union of connected components of the joint level set $u_1^{-1}(c_1)\cap\cdots\cap u_m^{-1}(c_m)$ of functions $u_r$ that satisfy the equations $T_ru_r=0$ in~$\RR^n$.
\end{theorem}

Our approach to the case where the prescribed joint level set has a noncompact component strongly depends on the method of proof of Theorem~\ref{T.1}, so the result will be stated in terms of the codimension-$m$ analog of a tentacled hypersurface, as defined in Section~\ref{S.compact}. To this noncompact submanifold we can also add a union of compact components of the kind dealt with in the previous theorem:

\begin{theorem}\label{T.2}
  Given an integer $m\geq2$, let $L$ be a tentacled submanifold of
  codimension $m$ in $\RR^n$ (possibly disconnected and of infinite
  type) and let $\tilde L$ be a finite union of pairwise disjoint,
  compact, codimension $m$ submanifolds of $\RR^n$ with trivial normal
  bundle. We assume that $L$ and $\tilde L$ are disjoint and fix $m$
  real constants $k_r$. Then the codimension-$m$ submanifold $L\cup
  \tilde L$ can be transformed by a smooth diffeomorphism $\Phi$ of
  $\RR^n$, arbitrarily close to the identity in the $C^1$ norm, so
  that $\Phi(L\cup\tilde L)$ is a union of connected components of the
  joint level set $u_1^{-1}(c_1)\cap \cdots \cap u_m^{-1}(c_m)$, each
  function $u_r$ satisfying the equation $(\De-k_r^2)u_r=0$ in
  $\RR^n$.
\end{theorem}

The proof of these results also rely on the combination of a robust local construction with a suitable global approximation. It is worth emphasizing that our treatment of noncompact joint level sets is considerably simplified by the fact that the diffeomorphisms that appear in Theorem~\ref{T.1} can be chosen $C^1$-small (and not only $C^0$-small). In proving this bound for the diffeomorphism, we make essential use of fine $C^2$-estimates for some Green's functions; had we restricted ourselves to the standard ($C^1$) gradient estimates, the diffeomorphisms would have been only $C^0$-small and the treatment of joint level sets would have been much more involved because the transversality of the intersection of hypersurfaces is generally not preserved when one applies a $C^0$-small diffeomorphism to each hypersurface. The lack of appropriate estimates for higher derivatives of the Green's functions also lies at the heart of why, in the noncompact case, we can make all the above diffeomorphisms $C^1$-small but not $C^p$-small for all $p$.

The paper is organized as follows. To begin with, in
Section~\ref{S.guide} we provide a guide to the demonstration, where
we present the global picture in a less technical manner and explain
the relationships between the different parts of the proof. In
Section~\ref{S.isotopy} we prove the novel noncompact stability
theorems that extend Thom's isotopy theorem to the noncompact
case. These results play a key role in the proof of Theorems~\ref{T.1}
and~\ref{T.2}, and are proved using a semi-explicit construction and
an appropriate gradient condition to deal with the lack of
compactness. In Section~\ref{S.approx} we prove the necessary global
approximation results that are needed in the rest of the paper,  which
generalize theorems of Gauthier, Goldstein and Ow~\cite{GGO83} and
Bagby~\cite{Ba88}. Once these crucial preliminary results have been
established, we present the proofs of Theorem~\ref{T.1} and
Corollary~\ref{C.algebraic} in Sections~\ref{S.finite}
and~\ref{S.infinite}, and those of Theorems~\ref{T.3} and~\ref{T.2} in
Section~\ref{S.compact}.

Summing up, the results we prove in this paper provide a satisfactory
understanding of the topology of the regular level sets of
solutions to certain elliptic PDEs, both scalar and vector-valued
(through the intersection of level sets of scalar solutions). Although
the main ideas of the proofs are similar, the class of elliptic
equations we can treat in each case is different, mainly because in
order to obtain fine control of the solutions at infinity (as in
Theorems~\ref{T.1} and~\ref{T.2}, but not in Theorem~\ref{T.3}) it is crucial
to consider more restricted classes of equations, which explains the
need for separate statements. For completeness, we have also included 
Appendix~\ref{S.appendix}, where we consider the case of one scalar
solution with a compact level set: while this situation cannot happen e.g.\ in the case of harmonic functions, which is the main thrust of the article, we show that any compact hypersurface is diffeomorphic to a level set of a solution of the equation $(\De-q)u=0$ in $\RR^n$ under fairly general hypotheses on the function $q$.

\section{Guide to the paper and strategy of proof}
\label{S.guide}

In this section we shall give the global picture of the proof of the main results, which will enable us to show the connections between the different parts of the proof without technicalities. We will also present the organization of the article, explaining the role that each section plays in the proof of the theorems. This section is divided into two parts, corresponding to the case of the level sets of a single function and to the case of joint level sets.

\subsection{Level sets of a single function}

Theorem~\ref{T.1} is motivated by Question~\ref{Q.1}: which surfaces are homeomorphic to a component of the zero set of a harmonic function in $\RR^3$? Hence, to explain the gist of our approach to this theorem, we can restrict ourselves to this particular case. For the sake of concreteness, we shall thus start by trying to prove that there exists a harmonic function in $\RR^3$ whose zero set has a connected component diffeomorphic to the torus of genus $g$ with $N$ ends, as displayed in Figure~\ref{F.jungle}.

The basic strategy to construct a harmonic function such that the above surface, which we shall call $L$, is diffeomorphic to a connected component of its zero set is the following. We will start with a function $v$ that is harmonic in a neighborhood of the surface and has a level set diffeomorphic to $L$. We will then approximate this local solution by a global harmonic function $u$. Hence, the key for the success of this strategy is to be able to ensure that the level set of the local harmonic function is `robust' in the sense that it is preserved, up to a diffeomorphism, under the perturbation corresponding to the above global approximation.

Let us elaborate on this point to show the kind of difficulties that
arise when carrying out this program. We have seen that the starting
point of the strategy must be a topological stability theorem for the
level sets of the local harmonic function $v$. If these level sets
were compact, this could be accomplished using Thom's celebrated
isotopy theorem~\cite[Theorem 20.2]{AR67}; however, the fact that the
level sets of a harmonic function are all noncompact makes the problem
much subtler (indeed, controlling bounded regions of level sets is
much easier, as we showed in Ref.~\cite{Annali} to study some pinching
and bending properties in harmonic function theory). Moreover, this
stability result must be finely tailored to provide sufficient control
of the deformation at infinity, for otherwise this part would be a
bottleneck in the proof of our theorems. Therefore, one of the crucial
parts of the article is the proof of a fine $C^1$ noncompact stability
result under small perturbations, which holds provided that the
function $v$ satisfies a suitable gradient condition in a {\em
  saturated}\/ neighborhood of the surface (roughly speaking, in
$v^{-1}((-\ep,\ep))$ in order to control the zero set of $v$) and a
$C^2$ bound.

In view of this result, the local harmonic function $v$ must be constructed in such a way that the above stability conditions hold. These conditions rule out, for example, an approach based on the Cauchy--Kowalewski theorem, as it does not yield enough information on the domain of definition of the solution. Instead, we base our construction on the use of Green's functions. We take an unbounded domain $\Om$ whose boundary is diffeomorphic to the surface $L$ and consider its Dirichlet Green's function $G_\Om(x,y)$. If $y$ is a point of this domain, $G_\Om(\cdot,y)$ defines a (local) harmonic function in a half-neighborhood of the boundary, which is the zero level set of this function and diffeomorphic to the surface $L$. However, the noncompact stability theorem cannot be applied to $G_\Om(\cdot,y)$ because this function does not satisfy neither the saturation nor the gradient conditions.

In order to circumvent this difficulty, we start by deforming the surface so that it has a controlled geometry at infinity (i.e., in this particular case, that the $N$ ends of the torus are straight cylinders, as in Figure~\ref{F.jungle}). For simplicity, we keep the notation $L:=\pd\Om$ for this deformation of the initial surface (this is the reason why, without loss of generality, we are requiring rigid ends in the definition of tentacled (hyper)surfaces, cf.\ Definitions~\ref{D.tentacled} and~\ref{D.infinite}). We then insert a straight half-line in each end of the domain and, denoting by $\mu_i$ the length measure on the $i$-th half-line, consider the function
\[
v(x):=\sum_{i=1}^N \int G_\Om(x,y)\,d\mu_i(y)\,.
\]
It can be shown that this defines a function harmonic in the domain minus the half-lines whose zero set is the `tentacled' torus $L$ and which does satisfy the boundedness, saturation and gradient conditions of the $C^1$ stability theorem in a half-neighborhood of the surface. Roughly speaking, the proof of this fact is based on estimates that exploit the asymptotic Euclidean symmetries of the construction and the exponential decay of the Green's function of the domain. As an aside, notice that to resort to this construction we will need to prove the noncompact stability theorem also for functions defined only in a half-neighborhood of the surface under consideration.

The last ingredient is the theorem that allows to approximate the local harmonic function $v$ by a global harmonic function in a certain neighborhood of the surface~$L$. This approximation must be small in the $C^1$ norm in order to apply the stability theorem. If $L$ were compact, the Lax--Malgrange theorem for elliptic PDEs~\cite[Theorem 3.10.7]{Na68} would provide an adequate global approximation result. As $L$ is noncompact, however, the situation is considerably more involved. The proof of our global approximation theorem relies on an iterative procedure (which does {\em not}\/ apply to arbitrary elliptic PDEs) that is built over an appropriate exhaustion by compact sets and combines the Lax--Malgrange theorem, suitable Green's function estimates and a balayage-of-poles argument.

The implementation of this strategy will be carried out in four stages:\smallskip

\noindent{\em Stage 1: Noncompact stability theorem.} In Section~\ref{S.isotopy} we prove a stability theorem for noncompact level sets (Theorem~\ref{T.stab}) and a variant for functions defined only in a half-neighborhood of the level set (Corollary~\ref{C.Thom}). Level sets of vector-valued function are also treated in view of their applications to the case of joint level sets, to be discussed below. When applied to compact submanifolds, this result fully recovers Thom's isotopy theorem.

\noindent{\em Stage 2: Global approximation theorem.} In Section~\ref{S.approx} we prove that a local solution of the equation $(\De-q)v=0$, with $q$ a real analytic function satisfying certain hypotheses, can be approximated in the $C^p$ norm by a global solution of the equation (Theorem~\ref{T.approx}). This extends results of Gauthier, Goldstein and Ow~\cite{GGO83} and Bagby~\cite{Ba88}, dealing with equations that obey a minimum principle but not a maximum principle.

\noindent{\em Stage 3: Noncompact level sets of finite topological
  type.} In Section~\ref{S.finite} we obtain a realization theorem for
noncompact hypersurfaces of finite topological type, that is, whose
homotopy groups are all finitely generated. We start by presenting the
definition of  tentacled hypersurfaces of finite type (Definition~\ref{D.tentacled}) and characterizing the hypersurfaces that can be realized as a tentacled hypersurface modulo diffeomorphism (Proposition~\ref{P.algebraic}). After a series of intermediate lemmas we show that for any tentacled hypersurface there exists a solution of the equation $(\De-k^2)u=0$ in $\RR^n$ which has a level set diffeomorphic to the hypersurface (Theorem~\ref{T.finite}). {\em Corollary~\ref{C.algebraic}}\/ follows from the latter theorem and Example~\ref{E.algebraic}.

\noindent{\em Stage 4: Noncompact level sets of infinite topological
  type.} In Section~\ref{S.infinite} we tackle the case of noncompact
hypersurfaces of infinite topological type
(Theorem~\ref{T.infinite}). To do so, we introduce the notion of periodic tentacled hypersurface (Definition~\ref{D.infinite}), discussing the associated discrete Euclidean symmetries that, together with the asymptotic Euclidean symmetries discussed above, play a key role in the proof of the theorem. {\em Theorem~\ref{T.1}}\/ then follows from Theorems~\ref{T.finite} and~\ref{T.infinite}; a further generalization is presented in Remark~\ref{R.monsters}.\smallskip

Since we have also considered level sets of the equation $(\De-k^2)u=0$, which can be compact for $k\neq0$, for completeness it is worth discussing the case of compact level sets. This is done in Appendix~\ref{S.appendix}, where we show that for any compact hypersurface there exists a global solution having a level set diffeomorphic to it (Theorem~\ref{T.appendix}). The proof is considerably less involved as one does not have to deal with the lack of compactness present in Theorems~\ref{T.1} and~\ref{T.2}, and actually applies as well to more general equations.

\subsection{Joint level sets}

We shall next deal with joint level sets of $m\geq2$ harmonic functions, that is, with the intersection of their level sets. As we did in the case of a single level set, we will be interested in the case of {regular} level sets, which amounts to requiring that this intersection be regular. This leads to the consideration of {\em transverse}\/ intersections of level sets $u_1^{-1}(c_1)\cap \cdots \cap u_m^{-1}(c_m)$, which means that the gradients of $u_1,\dots,u_m$ are linearly independent at each point of the joint level set. It is well known that the topological obstruction on the submanifolds that can be the transverse intersection of $m$ functions is that they must have codimension $m$ and trivial normal bundle.

Let us fist consider the case where the joint level set is compact, or, more generally, a locally finite disjoint union of compact submanifolds of codimension $m$ (with trivial normal bundle, by the above argument). The strategy we will follow to construct $m$ harmonic functions with a joint level set diffeomorphic to a given compact submanifold $L$ is similar to the one used for a single level set: we first construct $m$ local harmonic functions having a joint level set diffeomorphic to $L$ and then approximate them by global harmonic functions, using a topological stability theorem to guarantee that they have a joint level set that is also diffeomorphic to $L$.

To construct the local solutions, we make use of the triviality of the normal bundle of $L$ to characterize the submanifold as the transverse intersection of $m$ hypersurfaces, which can be chosen real analytic by perturbing $L$ a little if necessary. An easy application of the Cauchy--Kowalewski theorem then gives the desired local harmonic functions. The global approximation theorem we shall use is a better-than-uniform approximation result that we prove using an iterative scheme based on the Lax--Malgrange theorem. The key point that allows to simplify the treatment (and to consider more general elliptic PDEs, as in the statement of Theorem~\ref{T.3}) is that the compactness of each component of $L$ makes the verification of the topological stability conditions much less subtle than in the noncompact case; in particular, this is the reason why the above Cauchy--Kowalewski argument  is successful and topological stability can been tackled using only the classical Thom isotopy theorem.

The case of joint level sets with noncompact components is more involved and, as in the case of a single level set, requires some fine control on the geometry of the submanifold at infinity. Therefore, we introduce the notion of a codimension-$m$ tentacled submanifold of $\RR^n$, which is the proper analog of a tentacled hypersurface and is indeed given by a transverse intersection of $m$ such hypersurfaces. This automatically grants that the normal bundle of the submanifold is trivial.

The strategy now is to write the codimension-$m$ submanifold as the intersection of $m$ tentacled hypersurfaces of $\RR^n$. One can then apply the same reasoning we used in the proof of Theorem~\ref{T.1} to each of the aforementioned tentacled hypersurfaces, obtaining $m$ local solutions of the equations that have a level set diffeomorphic to one of the above tentacled hypersurfaces and satisfy the conditions of the noncompact stability theorem. Suitable control of the second-order derivatives of these local solutions allow us to proceed as in the case of a single level set using a global approximation theorem.

Our treatment of joint level sets, which we present in Section~\ref{S.compact},  will therefore consist of two parts:\smallskip

\noindent{\em Part 1: Compact joint level sets.} In Subsection~\ref{SS.1} we will prove that, given any locally finite disjoint union $L$ of compact submanifolds in $\RR^n$ of codimension $m$ and trivial normal bundle and a collection of real analytic elliptic differential operators of second order $T_r$, there are $m$ solutions of the equations $T_ru_r=0$ in $\RR^n$ such that $L$ is diffeomorphic to a joint level set $u_1^{-1}(c_1)\cap \cdots\cap u_m^{-1}(c_m)$, thereby establishing {\em Theorem~\ref{T.3}}. A key step in the proof is a better-than-uniform approximation theorem for these differential operators (Lemma~\ref{L.LM}).

\noindent{\em Part 2: Noncompact joint level sets.} In Subsection~\ref{SS.2} we will define the notion of a codimension-$m$ tentacled submanifold (Definition~\ref{D.joint}) and tackle the case of joint level sets with noncompact components, proving {\em Theorem~\ref{T.2}}.

\subsection*{Notation}

To conclude this section, let us present some notation that will be employed throughout the article. We will use the notation
\[
\|f\|_{C^p(U)}:=\max_{|\al|\leq p}\sup_{x\in U} |D^\al f(x)|
\]
for the $C^p$ norm of a map $f$ in a set $U\subset\RR^n$, and denote by $B(x,\rho)$ the ball of radius $\rho$ centered at a point $x\in\RR^n$. We will always assume that $n\geq3$. If $\Om$ is a (possibly unbounded) domain of $\RR^n$, we will denote by
\[
\la_\Om:=\inf\bigg\{\int |\nabla v|^2\,dx: v\in C^\infty_0(\Om),\; \int v^2\,dx=1\bigg\}
\]
the infimum of the $L^2$ spectrum of $-\De$ in the domain $\Om$ with Dirichlet boundary conditions.

As customary, we will say that a function satisfies a PDE in a closed
set $S$ if the PDE holds in some open set containing $S$. All the
diffeomorphisms appearing in this article are assumed $C^\infty$ and
orientation-preserving without further mention, and all the
submanifolds of $\RR^n$ that we will consider are $C^\infty$,
oriented, properly embedded and without boundary. We will say a
submanifold of $\RR^n$ has codimension $m$ if all its connected
components do.

\section{Noncompact stability theorem}
\label{S.isotopy}

In forthcoming sections, it will be essential to ensure the stability under small perturbations of certain level sets of various maps. When these level sets are compact and regular, this is granted by Thom's isotopy theorem~\cite[Theorem 20.2]{AR67}, but noncompact regular level sets are generally not stable.

Our goal in this section is to prove the following $C^p$ stability theorem for noncompact level sets, which is of separate interest and allows us to deal with uniform perturbations of the type considered in Section~\ref{S.approx}. Our proof is based on totally different ideas from those of Thom and has the crucial advantage of being explicit enough to allow for a fine control of the diffeomorphism at infinity. In our proof, the diffeomorphism is constructed essentially as the time-$1$ flow of a carefully chosen vector field, whose components are checked to be suitably small:

\begin{theorem}\label{T.stab}
  Let $U$ be a domain in $\RR^n$ and let $f:U\to\RR^m$ be a $C^\infty$ map, with $m<n$. Consider a (possibly unbounded) connected component $L$ of the zero set $f^{-1}(0)$ and suppose that:
\begin{enumerate}
\item There exists a domain $V$ whose closure is contained in $U$ and such that the component of $f^{-1}(B(0,\eta))$ connected with $L$ is contained in $V$ for some $\eta>0$.
\item The gradients of the components $f_r$ of the map $f$ satisfy the condition
\[
\inf_{x\in V}|\om_1\nabla f_1(x)+ \cdots+\om_m\nabla f_m(x)|>0
\]
for all vectors $\om\in\RR^m$ of unit norm.
\end{enumerate}
Then, given any $\ep>0$, there exists some $\de>0$ such that for any smooth function $g:U\to\RR^m$ with
\begin{equation}\label{supV}
\| f- g\|_{C^1(V)}<\de
\end{equation}
one can transform $L$ by a diffeomorphism $\Phi$ of $\RR^n$ so that $\Phi(L)$ is the intersection of the zero set $g^{-1}(0)$ with $V$. The diffeomorphism $\Phi$ only differs from the identity in a proper subset of $V$ and satisfies $\|\Phi-\id\|_{C^0(\RR^n)}<\ep$.

Moreover, if the $C^{p}$ norm $\|f-g\|_{C^{p}(V)}$ is small enough and the first $p+1$ derivatives of the maps $f$ and $g$ are bounded (i.e., $\|f\|_{C^{p+1}(V)}$ and  $\|g\|_{C^{p+1}(V)}$ are finite), then the diffeomorphism is $C^p$-close to the identity: $\|\Phi-\id\|_{C^p(\RR^n)}<\ep$.
\end{theorem}

Before presenting the proof of the theorem, a comment on its hypotheses is in order. Condition~(i) asserts that $V$ contains a neighborhood of the submanifold $L$ saturated by the map $f$. Condition~(ii) imposes a gradient bound on the function $f$ in the set $V$, and was also used by Rabier in his extension of Ehresmann's fibration theorem~\cite{Ra97} (notice, however, that the ideas we use in our proof are not related to Rabier's). When $m=1$, this condition simply asserts that $|\nabla f|\geq C>0$ in $V$, while when $m>1$ this condition measures to what extent the gradients of the components of the map $f=(f_1,\dots,f_m)$ remain linearly independent in the set $V$. 

\begin{proof}[Proof of Theorem~\ref{T.stab}]
  By~(i) we can assume that the open set $V$ is exactly $f^{-1}(B(0,\eta))$ and $U$ is a small neighborhood of the closure of $ V$ where condition~(ii) and~\eqref{supV} still hold. Let us take an open interval $I\supset[0,1]$ and define the map $h:I× U\to \RR× \RR^m$ as
  \begin{equation}\label{htx}
h(t,x):=\big(t,(1-t)f(x)+t\,g(x)\big)\,.
\end{equation}
Thus $h(t,x)$ is an auxiliary map which connects the functions $f$ and $g$ in the sense that $h(0,x)=(0,f(x))$ and $h(1,x)=(1,g(x))$.

Let us take an arbitrary point $x$ in $U$ and denote by $F$ the $m× n$ matrix $Df(x)$, with components $F_{ij}:=\pd_j f_i(x)$. Denoting by  $F^*$ the transpose matrix of $F$, by the condition~(ii) it is clear that the lower bound
\begin{equation}\label{F*}
|F^*\om|= \big|\om_1\nabla f_1(x)+\cdots+\om_m\nabla f_m(x)\big|\geq c_1|\om|
\end{equation}
holds for all vectors $\om\in\RR^m$, where $c_1$ is a positive constant independent of the point $x\in U$.

Our goal now is to show that the derivative of the map $h(t,x)$ is uniformly bounded from below. Consider the $(m+1)× (n+1)$ matrix $H:=\bar D h(t,x)$, which we write as $H=H_0+A$, with
\[
H_0:=\bigg(\begin{matrix}
  1 & 0\\
  0 & F
\end{matrix}\bigg)\,,\qquad A:=\bigg(\begin{matrix}
  0 & 0\\
  g(x)-f(x) & t\,(Dg(x)-F)
\end{matrix}\bigg)\,.
\]
Here we are using the notation $\bar D, \bar \nabla$ for the derivatives with respect to the variables $t$ and $x$, as opposed to the derivatives $D,\nabla$ with respect to $x$. Setting $\|A\|:=\sup_{|v|=1}|Av|$, it is clear that $\|A\|\leq C\de$ by the $C^1$ estimate~\eqref{supV}. Since
\[
|H^*v|\geq \big||H_0^*v|-|A^*v|\big|
\]
and $|A^*v|\leq\|A\||v|\leq C\de|v|$, it stems from the inequality~\eqref{F*} that
\begin{equation}\label{H*}
|H^*v|\geq \big(\!\min\{1,c_1\}-C\de\big)|v|\geq c_2|v|
\end{equation}
for small enough $\de$ and all $v\in\RR×\RR^m$, $(t,x)\in I× U$. Here $c_2$ is a positive constant independent of $t$ and $x$. Therefore, the kernel of $H^*$ is trivial, so $H$ is onto and thus the map $h$ (resp.\ $f$) is submersive in $I× U$ (resp.\ $U$).

Let us denote the components of the map $h$ as $(h_0,h_1,\dots,h_m)$, with $h_0(t,x)=t$ and $h_i(t,x)=(1-t)\,f_i(x)+t\,g_i(x)$ for $1\leq i\leq m$. The condition~\eqref{H*} ensures that the self-adjoint $(m+1)× (m+1)$ matrix $HH^*$ satisfies $HH^*\geq c_2^2$, so that its inverse $B:=(HH^*)^{-1}$ is positive definite and satisfies the inequality $B\leq c_2^{-2}$ for all $(t,x)\in I× U$. We will denote by $(B_{\mu\nu})_{\mu,\nu=0}^m$ the matrix elements of $B$.

In order to construct the desired diffeomorphism, for each integer $\mu=0,1,\dots,m$ it is convenient to consider the vector field $S_\mu$ defined by
\begin{equation}\label{Smu}
S_\mu:=\sum_{\la=0}^m B_{\mu\la}\,\bar \nabla h_\la(t,x)
\end{equation}
in $I× U$. By definition, the scalar product (in $\RR^{n+1}$) of these vector fields with the space-time gradient of the function $h_\nu$ is given by
\begin{align}\label{Smuhnu}
  S_\mu\cdot\bar \nabla h_\nu=\sum_{\la=0}^m B_{\mu\la} \,\bar \nabla h_\la\cdot \bar \nabla h_\nu= \sum_{\la=0}^m\sum_{\be=0}^n B_{\mu\la} H_{\la \be} H_{\nu \be}=\de_{\mu\nu}\,,
\end{align}
with $\de_{\mu\nu}$ standing for Kronecker's delta. Hence these vector fields define a parallelism; in particular, each field $S_\mu$ does not vanish in $I× U$. Besides, denoting by $(S_{\mu 0},S_{\mu 1},\dots, S_{\mu n})$ the components of $S_\mu$, one can readily check that the component $S_{\mu \be}$ coincides with the matrix element $(BH)_{\mu \be}$, so that the norm of these vector fields is bounded in $I× U$ by
\begin{equation}\label{Smu}
|S_\mu|^2=\sum_{\be=0}^n S_{\mu \be}^2= B_{\mu\mu}\leq c_2^{-2}\,.
\end{equation}

Choosing a small enough positive constant $\de<\frac\eta{4}$ and using the condition~\eqref{supV}, we can ensure that $g^{-1}(0)$ is contained in the proper subset $W:=f^{-1}(B(0,\frac\eta2))$ of $U$. Let us take a $C^\infty$ function $\chi:I× U\to[0,1]$ which is equal to $1$ in $[0,1]× W$ and is supported in $I× U$. 

From Eq.~\eqref{Smuhnu} with indices $\mu=\nu=0$ it stems that one can decompose the vector field $S_0$ as
\begin{equation}\label{eequis}
S_0=:\pd_t+X\,,
\end{equation}
where $X$ is a vector field in $I× U$ orthogonal to the coordinate vector field $\pd_t$. Let us now consider the vector field $\hat S_0:=\pd_t+\chi\,X$, whose time-$1$ flow will essentially yield the desired diffeomorphism. Since $|\hat S_0|\leq|S_0|$ is bounded in $I× U$ by the estimate~\eqref{Smu} and coincides with the coordinate field $\pd_t$ in a neighborhood of the boundary $\pd(I× U)$, it follows that the vector field $\hat S_0$ defines a $C^\infty$ local flow $\vp_s$ in $I× U$. Let us write this local flow as
\begin{equation}\label{hPhi}
\vp_s(t,x)=:\big(t+s,\hPhi(s,t,x)\big)
\end{equation}
and notice that, for any $x\in U$, $\vp_s(0,x)$ is well defined at least for $s\in[0,1]$.

Let us define the codimension-$(m+1)$ submanifolds
\[
L_0:=\{0\}× f^{-1}(0)\,,\qquad L_1:=\{1\}× g^{-1}(0)
\]
of the set $I× W$. Clearly $L_0=h^{-1}(0,0)$ and $L_1=h^{-1}(1,0)$. We claim that the time-$s$ flow $\vp_s(L_0)$ of $L_0$ is contained in the time slice $\{s\}× W$ for all $s\in[0,1]$. In order to see this, let us take $(0,x_0)\in L_0$ and use the notation $(s,x(s)):=\vp_s(0,x_0)$. As $S_0\cdot \bar\nabla h_i=0$ for $1\leq i\leq m$ by Eq.~\eqref{Smuhnu}, one finds that
\begin{align*}
  \bigg|\frac{d}{ds}h_i(s,x(s))\bigg|&=(1-\chi)\big|X\cdot \nabla h_i(s,x(s))\big| = (1-\chi)\bigg| \frac{\pd h_i}{\pd t}(s,x(s))\bigg|\\
  &=(1-\chi)\big|f_i(x(s))-g_i(x(s))\big|<\de
\end{align*}
by the estimate~\eqref{supV}, so that $|h_i(s,x(s))|<\de$ for all $s\in[0,1]$. In turn, this yields the bound
\[
|f_i(x(s))|<2 \de
\]
for all $s\in[0,1]$ and $1\leq i\leq m$, so that the time-$s$ flow  $\vp_s(L_0)$ is contained in $[0,1]× W$ for $s\in[0,1]$. Since $\chi=1$ in $[0,1]× W$, actually $h_i(s,x(s))=h_i(0,x_0)=0$, which implies that $\vp_1(L_0)\subset L_1$. By reversing the argument one infers that $\vp_{-1}(L_1)\subset L_0$, so that $\vp_1(L_0)=L_1$, as claimed.

Let us now consider the map $\Phi_1(x):=\hPhi(1,0,x)$ (see Eq.~\eqref{hPhi}), which is a diffeomorphism of $U$ by construction. As the map $\Phi_1:U\to U$ so defined is then equal to the identity in a neighborhood of $\pd U$, it can be trivially extended to a $C^\infty$ diffeomorphism $\Phi$ of $\RR^n$ by setting
\[
\Phi(x):=\begin{cases}
  \Phi_1(x)\,, & x\in U\,,\\
  x\,, &x\not\in U\,.
\end{cases}
\]
By construction, $\Phi(L)$ coincides with $g^{-1}(0)\cap V$.

To complete the proof of the theorem, it only remains to show that the diffeomorphism $\Phi$ is  close to the identity. Since $|\chi|\leq 1$ and $\Phi_1$ is defined in terms of the local flow~\eqref{hPhi}, in order to show that the norm $\|\Phi-\id\|_{C^0(\RR^n)}$ is small it suffices to prove that $\|X\|_{C^0(I× U)}$ is close to zero, where $X$ is the vector field defined in Eq.~\eqref{eequis}. Since the space-time gradient of $h_0$ is $\pd_t$, it follows from the definition of $X$ and Eq.~\eqref{Smuhnu} that $X$ can be written as a linear combination of the gradients of $h_\mu$:
\[
X=\sum_{\mu=0}^m v_\mu\,\bar \nabla h_\mu\,.
\]
Here the function $v_\mu=v_\mu(t,x)$ depends on the variables $t$ and $x$. Notice, moreover, that
\[
X\cdot \bar \nabla h_\mu=S_0\cdot \bar \nabla h_\mu-\frac{\pd h_\mu}{\pd t}=\begin{cases}0\,, & \mu=0\,,\\
f_\mu-g_\mu\,, & \mu\in\{1,\dots, m\}
\end{cases}
\]
by Eq.~\eqref{Smuhnu}. Hence the norm of the vector field $X$ satisfies
\begin{align*}
|X|&= \frac{|X\cdot\sum_{\mu=0}^m v_\mu\,\bar \nabla h_\mu|}{|\sum_{\mu=0}^m v_\mu\,\bar \nabla h_\mu|}\leq\frac{\sum_{i=1}^m |v_i|\,|f_i-g_i|}{|H^*v|}\leq \frac{m|v|\sup_U|f-g|}{c_2|v|}<\frac{m\de}{c_2}
\end{align*}
by Eqs.~\eqref{supV} and~\eqref{H*}, thereby proving the smallness of the $C^0$ norm $\|\Phi-\id\|_{C^0(\RR^n)}$.

To estimate the $C^p$ norm of $\Phi-\id$, one argues as in the case of the $C^0$ norm but one needs to estimate the $C^p$ norm of the vector field $\chi X$. First of all, it should be noticed that the hypothesis that the derivatives of $f$ are bounded ensures that the set $V$ contains a `metric tubular neighborhood'
\[
\big\{ x\in\RR^n: \dist(x,L)<\rho\big\}
\]
of the submanifold $L$, where $\rho$ is some positive constant. This obviously allows us to assume that all the derivatives of the function $\chi$ are bounded in $\RR^n$. After a straightforward but tedious computation, one finds that the $C^p$ norm of the vector field $X$ can be controlled as
\[
|D^p X|\leq C(\|f\|_{C^{p+1}},\|g\|_{C^{p+1}}, \|\chi\|_{C^{p}})\,\|f-g\|_{C^p}
\]
in the set $I× U$, where $ C(\|f\|_{C^{p+1}},\|g\|_{C^{p+1}}, \|\chi\|_{C^{p}})$ is a constant depending on the norms of $f$, $g$ and $\chi$. The theorem then follows.
\end{proof}

\begin{remark}\label{R.Thom}
  When the submanifold $L$ is compact, all the derivatives of the maps $f$ and $g$ are necessarily finite in a suitable neighborhood of $L$, so the $C^p$ condition $\|f-g\|_{C^p(V)}<\de$ automatically yields a diffeomorphism $\Phi$ that is close to the identity in the $C^p$ norm (as was proved in Thom's isotopy theorem).
\end{remark}

As a matter of fact, it is worth pointing out that the above proof does not only demonstrate Theorem~\ref{T.stab}, but also the following closely related result, in which the component $L$ of the zero set of the function is allowed to lie on the boundary of the set $V$. Although this is only a minor modification, we will find this result of use in forthcoming sections:

\begin{corollary}\label{C.Thom}
Let $U$ be a domain in $\RR^n$ and let $f:U\to\RR^m$ be a $C^\infty$ map, with $m<n$. Consider a (possibly unbounded) connected component $L$ of $f^{-1}(0)$. Suppose that:
\begin{enumerate}
\item There exists a domain $V$ whose closure is contained in $U$ and such that the intersection $f^{-1}(0)\cap \overline V$ is $L$.
\item The gradients of the components $f_r$ of the map $f$ satisfy the condition
  \[
  \inf_V|\om_1\nabla f_1+\cdots+ \om_m\nabla f_m|>0
\]
  for all unit vectors $\om\in\RR^m$.
\end{enumerate}
Then, given any $\ep>0$ there exists some $\de>0$ such that for any smooth function $g:U\to\RR^m$ with $\|
f- g\|_{C^1(V)}<\de$
one can transform $L$ by a diffeomorphism $\Phi$ of $\RR^n$ so that $\Phi(L)$ is the intersection of the zero set  $g^{-1}(0)$ with the closure of $V$ provided that any component of the set
\begin{equation}\label{sett}
\big\{x\in U:(1-t)\,f(x)+t\,g(x)=0\big\}
\end{equation}
that intersects $\overline V$ is contained in $\overline V$ for all $t\in[0,1]$. Besides $\Phi-\id$ is supported in a proper subset of any fixed neighborhood of $\overline V$ and $\|\Phi-\id\|_{C^0(\RR^n)}<\ep$. If moreover the $C^p$ norm $\|f-g\|_{C^p(V)}$ is small enough and the first $p+1$ derivatives of the maps $f$ and $g$ are bounded in $V$, then one can take the diffeomorphism $\Phi$ $C^p$-close to the identity.
\end{corollary}

\begin{remark}
  In this paper we will only apply this corollary to $g:=f-c$, in which case the sets~\eqref{sett} are simply $f^{-1}(tc)$. Clearly a straightforward modification of Eq.~\eqref{htx} allows to deal with the case where the level sets $f^{-1}(tc)$ do not have the above property but there is a smooth curve $\ga:[0,1]\to B(0,\de)$ with $\ga(0)=0$, $\ga(1)=c$ and such that any component of $f^{-1}(\ga(t))$ meeting $\overline V$ is contained in $\overline V$. An analogous result can also be established for the case of two arbitrary maps $f$ and $g$.
\end{remark}

\section{A global approximation theorem}
\label{S.approx}

Our goal in this section is to establish that any function $v$ satisfying $(\De-k^2) v=0$ in a closed (possibly unbounded) set $S\subset\RR^n$ can be approximated in the $C^p$ norm in this set by a global solution of the latter equation, provided that the complement of the set $S$ does not have any bounded connected components. This condition on $S$ is used to apply the Lax--Malgrange theorem and in an argument on the balayage of poles.

We will prove this result for the equation 
\begin{equation*}
(\De-q) v= 0\qquad \text{in }S\subset\RR^n\,,
\end{equation*}
as the proof for $(\De-k^2)u=0$ is not substantially easier. Here $q$
is a bounded, nonnegative function on
$\RR^n$ which we take real analytic (this condition is
not necessary, as all the arguments in this section can be modified to deal with lower
regularity, but this allows us to avoid inessential technicalities). In order to get a $C^p$ uniform approximation result, we will
assume that the $C^{p-2}$ norm $\|q\|_{C^{p-2}(\RR^n)}$ is finite. The
results we prove in this section generalize theorems of Gauthier,
Goldstein and Ow~\cite{GGO83} and Bagby~\cite{Ba88} to equations that
satisfy a minimum principle but not a maximum principle. 

Throughout the paper, we shall often need the Dirichlet Green's function $G_\Om(x,y)$ of the operator $\De-q$ in an unbounded domain $\Om$. For completeness, in the following proposition we prove the existence of a minimal Dirichlet Green's function using the classical method of compact exhaustions and summarize some properties that will be required later. When $\Om=\RR^n$, we will write $G(x,y)$ instead of $G_{\RR^n}(x,y)$ for the ease of notation.

\begin{proposition}\label{P.Green}
Let $\Om$ be a domain in $\RR^n$, possibly unbounded and with nonempty boundary of class $C^\infty$. Then there exists a minimal positive Dirichlet Green's function 
\[
G_\Om:(\Om× \Om)\minus\diag(\Om× \Om)\to\RR\,,
\]
which satisfies:
\begin{enumerate}
  \item The Green's function is symmetric, that is, $G_\Om(x,y)=G_\Om(y,x)$.
  \item $G_\Om(\cdot,y)\in C^\infty(\BOm\minus\{y\})\cap L^1\loc(\Om)$ and $(\De_x-q(x))G_\Om(x,y)=-\de_y(x)$, with $\de_y$ being the Dirac measure supported at $y$.  
  \item $G_\Om(\cdot,y)=0$ on the boundary $\pd\Om$ and $G_\Om(x,y)\leq C|x-y|^{2-n}$, with $C^{-1}=(n-2)|\SS^{n-1}|$.
\end{enumerate}
\end{proposition}
\begin{proof}
  If $\Om$ is bounded, the result is well known, so we henceforth assume that $\Om$ is unbounded.  Let $\Om_1\subset\BOm_1\subset \Om_2\subset\cdots\subset\BOm$ be an exhaustion of $\BOm$ by bounded domains with smooth boundaries and let
\[
G_i:(\Om_i× \Om_i)\minus \diag(\Om_i× \Om_i)\to\RR
\]
be the Dirichlet Green's function of the domain $\Om_i$, which is symmetric and satisfies
\[
\big(\De_x-q(x)\big)G_i(x,y)=-\de_y(x)
\]
in $\Om_i$ and $G_i(\cdot,y)|_{\pd\Om_i}=0$.

Since $q\geq0$, the minimum principle ensures that the Green's function $G_i(\cdot,y)$ is positive in $\Om_i$ and that the sequence of Green's functions is monotonically increasing in the sense that
\begin{equation}\label{monotone}
G_i(x,y)< G_{i+1}(x, y)
\end{equation}
for all $x,y\in\Om_i$. Moreover, the minimum principle also guarantees that $G_i$ is bounded by the Green's function of the Laplacian in $\RR^n$ as
\[
G_i(x,y)\leq C|x-y|^{2-n}\,,
\]
with $C^{-1}:=(n-2)|\SS^{n-1}|$. By the monotonicity property~\eqref{monotone}, $\big(G_i(x,y)\big)_{i=1}^\infty$ is then a Cauchy sequence. Standard gradient estimates~\cite{GT98} imply that, for any open sets $B\subset\overline B\subset B'\subset\overline{B'}\subset (\Om_i\cap\Om_j)\minus\{y\}$,
\[
\| G_i(\cdot,y)- G_j(\cdot,y)\|_{C^2(B)}< C'\|G_i(\cdot,y)- G_j(\cdot,y)\|_{C^0(B')}\,,
\]
where $C'$ only depends on  $B$, $B'$ and on the function $q$. Hence $G_i(\cdot,y)$ converges $C^2$-uniformly on compact subsets of $\Om\minus\{y\}$ to a solution $G_\Om(\cdot,y)$ of the equation $\big(\De-q\big)G_\Om(\cdot,y)=-\de_y$. The fact that $G_\Om$ has the properties (i)--(iii) stems from the above construction.
\end{proof}

The proof of the approximation theorem relies on the following three lemmas, whose proofs make use of several ideas of Bagby~\cite{Ba88} and Bagby and Gauthier~\cite{BG88} but rely on a different argument because the differential equation $(\De-q)v=0$ does not satisfy a maximum principle. Lemma~\ref{L.G} estimates the $C^p$ norm of the difference between Dirichlet Green's functions with nearby poles using Schauder estimates (throughout this section, $p$ will denote an arbitrary positive integer):

\begin{lemma}\label{L.G}
Let $V$ be an open subset of $\RR^n$ and let $y$ be a point in $V$. Then for any $\ep>0$ there exists an open neighborhood $B_y\subset V$ of $y$ such that
\[
\|  G(\cdot,y) - G(\cdot,z)\|_{C^p(\RR^n\minus V)}<\ep
\]
for all $z\in B_y$.
\end{lemma}
\begin{proof}
  We can assume that $V$ is bounded without loss of generality. Let us arbitrarily fix some real constant $\rho>0$. Since $(\De-q)G(\cdot, z)=0$ in $\RR^n\minus\{z\}$, interior Schauder estimates~\cite[Corollary 6.3]{GT98} show that there exists a constant $C_1$ (depending solely on $n$, $p$, $\rho$ and $\|q\|_{C^{p-2}(\RR^n)}$) such that the $C^p$ pointwise estimate
\begin{equation}\label{DalG}
\max_{|\al|\leq p}\big|D^\al_x G(x, z)\big|\leq C_1\|G(\cdot, z)\|_{C^0(B(x,\rho))}\leq C_2\big(|x-z|-\rho\big)^{2-n}
\end{equation}
holds whenever $|x-z|>\rho$. Here the last inequality follows from the third statement in Proposition~\ref{P.Green}, thereby showing that $D^\al G(\cdot,z)$ tends to zero at infinity for all $|\al|\leq p$. 

By the $C^p$ estimate~\eqref{DalG} and the boundedness of the set $V$, for any $\ep>0$ one can take a compact subset $K\supset V$ of $\RR^n$ such that 
\begin{equation}\label{Brho}
\sup_{z\in V}\sup_{x\in \RR^n\minus K}\big|D^\al_x G(x, z)\big|<\frac\ep2
\end{equation}
for $|\al|\leq p$. Moreover, since $D^\al_x G(x, z)$ depends continuously on $z\in\RR^n\minus\{x\}$ and $K$ is compact, there is a small neighborhood $B_y\subset V$ of $y$ such that 
\[
\sup_{z\in B_y}\sup_{x\in K\minus V}\big|D^\al_x G(x, y)-D^\al_x G(x, z)\big|<\ep\,.
\]
By the definition of the set $K$, the lemma follows.
\end{proof}

In the following lemma we show that the Green's function with pole $y$ can be approximated in a suitable sense by finite linear combinations of Green's functions with poles in a prescribed set. The proof is based on density and duality arguments:

\begin{lemma}\label{L.approx}
Given an open set $W\subset \RR^n$, let us take a compact subset $K\subset W$ of nonempty interior and a point $y\in W$. Then, for any $\ep>0$, there exist a finite set $\cP\subset K$ and constants $(a_z)_{z\in \cP}\subset\RR$ such that
\[
\bigg\|G(\cdot,y)-\sum_{z\in\cP}a_z\,G(\cdot,z)\bigg\|_{C^p(\RR^n\minus W)}<\ep
\]
\end{lemma}
\begin{proof}
We assume that the point $y$ does not belong to the set $K$, since otherwise the statement is trivial, and consider a proper open subset $W'\subset W$ containing $y$ and $K$, with smooth boundary and such that $\rho:=\frac12\dist(\pd W,\pd W')>0$. (Later on, we will take advantage of this set to transform a $C^0$ estimate into a $C^p$ estimate.) Let $C_0(\RR^n\minus W')$ denote the Banach space of continuous functions on the complement of $W'$ tending to $0$ at infinity, endowed with the supremum norm. We denote by $\cK$ the subspace of $C_0(\RR^n\minus W')$ consisting of all finite linear combinations of $G(\cdot,z)$ with $z\in K$. 

By the generalized Riesz--Markov theorem, the dual of $C_0(\RR^n\minus W')$ is  the space $\cM(\RR^n\minus W')$ of the finite signed Borel measures on $\RR^n$ whose support is contained in the complement of $W'$. Let us take any measure $\mu\in\cM(\RR^n\minus W')$ such that $\int f\,d\mu=0$ for all $f\in \cK$. Let us now define a function $F\in L^1\loc(\RR^n)$ by
\[
F(x):=\int G(z,x)\,d\mu(z)\,,
\]
so that $F$ satisfies the equation $(\De-q)F=-\mu$. Since $F$ is identically zero on the set $K$ by the definition of $\mu$ and $K$ has nonempty interior, the unique continuation theorem ensures that the function $F$ vanishes on $W'$. It then follows that $\mu$ also annihilates $G(\cdot,y)$ because
\[
F(y)=\int G(z,y)\,d\mu(z)=0\,,
\]
which shows that $G(\cdot,y)$ can be uniformly approximated on $\RR^n\minus W'$ by elements of the subspace $\cK$ as a consequence of the Hahn--Banach theorem.

To complete the proof of the theorem, let us take a sequence $(f_i)_{i=1}^\infty\subset\cK$ such that
\[
\|G(\cdot,y)-f_i\|_{C^0(\RR^n\minus W')}<\frac1i\,,
\]
which is guaranteed to exist by the above argument.  We can now use the Schauder estimate~\eqref{DalG} to show that, for all $x\in\RR^n\minus W$, one has
\[
\max_{|\al|\leq p}\big|D_x^\al G(x,y)-D^\al f_i(x)\big|\leq C_1\|G(\cdot,y)-f_i\|_{C^0(B(x,\rho))}<\frac{C_1}i
\]
with $C_1$ independent of $x$, as claimed.
\end{proof}

In the following key lemma we discuss the balayage of the poles, which basically allows us to get rid of the poles by taking them to infinity, exploiting the fact that the set $V$ is unbounded. The proof relies on an iterative argument that utilizes Lemma~\ref{L.approx} to sweep the poles further in each step.

\begin{lemma}\label{L.balayage}
Let $V$ be an unbounded domain of $\RR^n$ containing a point $y$. Then, for any $\ep>0$, there is a function $w $ which satisfies the equation $(\De-q)w =0$ in $\RR^n$ and approximates $G(\cdot,y)$ as
\[
\|G(\cdot,y)- w \|_{C^p(\RR^n\minus V)}<\ep\,.
\]
\end{lemma}
\begin{proof}
As the set $V$ is unbounded, we can take a parametrized curve $\ga:[0,\infty)\to V$ without self-intersections (that is, $\ga$ is injective) such that
\[
\ga(0)=y\,,\qquad \lim_{t\to\infty}|\ga(t)|=\infty\,.
\]
This is the curve that we will use to `sweep the pole to infinity'.
For each nonnegative integer $i$ we denote by $K_i\subset V\cap B(\ga(i),\frac12)$ a compact neighborhood of the point $\ga(i)$, and we let $(W_i)_{i=1}^\infty\subset V$ be a family of bounded domains such that each $W_i$ contains both $K_{i-1}$ and $K_{i}$. We also require that $W_i$ is `narrow enough' in the sense that
\[
\sup_{x\in W_i}\dist\big(x,\ga([i-1,i])\big)<1\,.
\] 
Moreover, we introduce the notation $\cK_i$ for the space of finite linear combinations of $G(\cdot,z)$ with poles $z\in K_i$.

Let us fix $\ep>0$. By Lemma~\ref{L.approx}, there exists $w_1\in\cK_1$ such that
\[
\| G(\cdot,y)- w_1\|_{C^p(\RR^n\minus W_1)}<\frac\ep2\,.
\]
Since $w_1$ is a finite linear combination of Green's functions with poles in $K_1$, using Lemma~\ref{L.approx} one can inductively show that there are functions $w_i\in\cK_i$ such that
\begin{equation}\label{epi}
\|w_i -w_{i-1}\|_{C^p(\RR^n\minus W_i)}<\frac\ep{2^{i}}\,,
\end{equation}
for all $i\geq2$. Since the distance between the point $y$ and the set $W_i$ tends to infinity as $i\to\infty$, this shows that $w_i$ converges $C^p$-uniformly on compact sets to a function $w $ that solves the equation $(\De-q)w =0$ in $\RR^n$. The lemma now follows upon noticing that
\begin{align*}
\big\| G(\cdot,y)- w_i\big\|_{C^p(\RR^n\minus V)}&\leq\big\| G(\cdot,y)- w_1\big\|_{C^p(\RR^n\minus V)}+\sum_{j=2}^i\big\| w_j- w_{j-1}\big\|_{C^p(\RR^n\minus V)}\\
&<\sum_{j=1}^i\frac{\ep}{2^j}<\ep
\end{align*}
for all $i$, on account of the $C^p$ estimate~\eqref{epi}.
\end{proof}

Armed with the previous lemmas, we are now ready to prove the main result of this section, which will be frequently used in the rest of the paper. The rough idea of the proof is to extend the local solution $v$ to a smooth function $w$ in $\RR^n$ supported in a neighborhood of $S$; using an iterative procedure based on Lemmas~\ref{L.approx} and~\ref{L.balayage}, $w$ is then approximated by a linear combination of Green's functions with poles in the complement of $S$ and subsequently the poles are swept off to infinity.

\begin{theorem}\label{T.approx}
Let $S$ be a closed subset of $\RR^n$ whose complement does not have any relatively compact connected components. Then any function $v$ that satisfies the equation  $(\De-q)v=0$ in $S$ can be approximated in the $C^p$ norm by a global solution to this equation. (That is to say, for any $\ep>0$ there exists a function $u$ satisfying $(\De-q)u=0$ in $\RR^n$  with $\|u -  v\|_{C^p(S)}<\ep$.)
\end{theorem}
\begin{proof}
By hypothesis, there exists an open subset $\Om\supset S$ such that $(\De-q)v=0$ in $\Om$. Let us take a smooth function $\chi:\RR^n\to\RR$  equal to $1$ in a neighborhood of $S$ and identically zero outside $\Om$, and define a smooth extension of $v$ to $\RR^n$ by setting $w:=\chi v$.

Consider an exhaustion of $\RR^n$ by bounded domains $\Om_i$ as in the proof of Proposition~\ref{P.Green}, which we choose so that $\RR^n\minus\Om_i$ does not have any compact components. We define the sets $U_i:=\Om_{i+1}\minus\BOm_{i-1}$, with $\Om_0:=\emptyset$, and take smooth functions $\vp_i$ on $\RR^n$ whose support is contained in $U_i$ and such that $\sum_{i=1}^j\vp_i=1$ in a neighborhood of $\Om_j$. A possible way to construct these functions is to take smooth functions $\phi_i$ compactly supported in $\Om_{i+1}$ and such that $\phi_i=1$ in a neighborhood of $\Om_i$; one can now set $\vp_i:=\phi_i-\phi_{i-1}$, with $\phi_0:=0$.
 
We will now write the function $w$ as a sum of global solutions to the equation and of functions that satisfy the equation but in a compact set. To do so, for each positive integer $i$ let us define the smooth function
\[
f_i:=\vp_i(q-\De)w\,,
\]
which is supported in $(\Om\minus S)\cap U_i$ by the definition of the functions $\chi$ and $\vp_i$. Defining the function 
\[
w_i(x):=\int G(x,y)\, f_i(y)\,dy\,,
\]
it is apparent that $w_i$ satisfies the equation $(\De-q)w_i=-f_i$, so that
\[
(\De-q)\bigg(w-\sum_{i=1}^jw_i\bigg)=0
\]
in $\BOm_j$. By applying the Lax--Malgrange theorem~\cite[Theorem 3.10.7]{Na68} inductively, one  easily infers that there exist functions $g_i$ satisfying the equation $(\De-q)g_i=0$ in $\RR^n$ and such that
\begin{equation}\label{BOmj}
\bigg\| w-\sum_{i=1}^j (w_i+g_i)\bigg\|_{C^p(\BOm_j)}<\frac1j
\end{equation}
for any positive integer $j$. Therefore the function $w$ can be expressed as
\begin{equation}\label{w.conv}
w=\sum_{i=1}^\infty(w_i+g_i)\,,
\end{equation}
the convergence being $C^p$-uniform on compact subsets of $\RR^n$.

Let us now approximate $w_i$ by functions $\hw_i$ that satisfy the equation in $\RR^n$ minus a finite set of points. Let us use the notation $\|f_i\|_{L^1}:=\int |f_i(x)|\,dx$ for the $L^1$ norm of $f_i$ and suppose that this quantity is nonzero. By Lemma~\ref{L.G}, for each point $y$ in the support of the function $f_i$
one can find a neighborhood $B_y\ni y$ contained in $(\RR^n\minus S)\cap U_i$ such that
\[
\| G(\cdot,y) -  G(\cdot,z)\|_{C^p(S\cup(\RR^n\minus U_i))}<\frac\ep{\|f_i\|_{L^1}2^{i+1}}
\]
for all  $z\in B_y$. As the support of $ f_i$ is compact, there exist finite pairwise disjoint sets $\cP_i\subset \supp f_i$ such that $\supp f_i$ is contained in the union of balls $\bigcup_{y\in\cP_i}B_{y}$. Let $(\chi_y)_{y\in\cP_i}\subset C^\infty_0(\RR^n)$ be a partition of unity subordinated to the balls $\{B_y:y\in\cP_i\}$, that is, nonnegative functions such that $\supp\chi_y\subset B_y$ and
\begin{equation*}
  \sum_{y\in\cP_i}\chi_y(x)=1
\end{equation*}
for all $x$ in a neighborhood of $\supp f_i$. Defining the function $\hw_i$ as
\[
\hw_i(x):=\sum_{y\in\cP_i}G(x,y)\,\int\chi_y(z)\, f_i(z)\,dz
\]
and using that $\chi_y(x)$ defines a partition of unity, one easily obtains the $C^p$ estimate
\begin{align}
\big| D^\al w_i(x)- D^\al \hw_i(x)\big|&=\bigg|\sum_{y\in\cP_i} \int\big( D^\al_x G(x,z)- D^\al_x G(x,y)\big)\, \chi_y(z)\, f_i(z)\,dz\bigg|\notag\\
&< \frac{\ep}{\|f_i\|_{L^1} 2^{i+1}}\int |f_i(z)|\,\bigg(\sum_{y\in\cP_i}\chi_y(z)\bigg)\,dz=\frac{\ep}{2^{i+1}}\label{ulth1}\,,
\end{align}
which holds pointwise in $S\cup(\RR^n\minus U_i)$ for all $|\al|\leq p$. When $\|f_i\|_{L^1}=0$, we can simply set $\hw_i:=0$.

We will next use a balayage argument to sweep the poles of $\hw_i$. As $\RR^n\minus S$ does not have any bounded connected components, one can take (possibly disconnected) pairwise disjoint, open unbounded sets $V_i\subset \RR^n\minus S$ such that each point $y\in\cP_i$ is contained in an unbounded component of $V_i$. Notice that the sets $V_i$ can be chosen so that for each compact set $K\subset\RR^n$ there exists an integer $J$ such that $K\cap V_i=\emptyset$ for all $i\geq J$. As $\hw_i$ is a finite linear combination of Green's functions with poles in $\supp f_i\subset U_i\minus S$, from Lemma~\ref{L.balayage} it follows that there exists a function $h_i$ satisfying the equation $(\De-q)h_i=0$ in $\RR^n$ and such that
\begin{equation}\label{ulth2}
 \|\hw_i- h_i\|_{C^p(\RR^n\minus V_i)}<\frac{\ep}{2^{i+1}}\,.
\end{equation}

As we will see, this condition ensures that the function
\begin{equation}\label{u.conv}
u:=\sum_{i=1}^\infty(h_i+g_i)
\end{equation}
is well defined and that the sum converges $C^p$-uniformly on compact subsets of $\RR^n$. In order to show this, let us take an arbitrary compact subset $K$ of $\RR^n$ and an integer $J$ such that $V_i\cap K=U_i\cap K=\emptyset$ for all $i\geq J$, which is known to exist by the way the sets $V_i$ and $U_i$ have been defined. In particular, the latter condition implies that $K$ is contained in $\BOm_{J-1}$. The $C^p$-uniform convergence of the series~\eqref{u.conv} on the set $K$ then follows from the pointwise $C^p$ estimate on $K$
\begin{align*}
  \bigg|\sum_{i=J}^k \big(D^\al h_i &+D^\al g_i \big)\bigg|\leq \sum_{i=J}^k \big|D^\al h_i -D^\al \hw_i \big|+ \sum_{i=J}^k \big|D^\al \hw_i -D^\al w_i \big| \\
  &\qquad\qquad\qquad\qquad\qquad\qquad\qquad\qquad\qquad+ \bigg|\sum_{i=J}^k \big(D^\al g_i
  +D^\al w_i \big)\bigg|\\
  &\leq \big|D^\al w \big| +\bigg|D^\al w -\sum_{i=1}^k\big(D^\al g_i +D^\al w_i \big)\bigg| +\bigg|\sum_{i=1}^{J-1}\big(D^\al g_i + D^\al w_i \big)\bigg|+\ep\,,\\
  &\leq  \big|D^\al w \big|+\bigg|\sum_{i=1}^{J-1}\big(D^\al g_i + D^\al w_i \big)\bigg|+\ep+\frac1k\,,
  \end{align*}
where $|\al|\leq p$, $k>J$ and we have used Eqs.~\eqref{BOmj}--\eqref{ulth2}. As a consequence of the $C^p$-uniform convergence on compact sets, clearly $u$ satisfies the equation $(\De-q)u=0$ in $\RR^n$.

Finally, one can invoke the uniform convergence of the series~\eqref{w.conv} and~\eqref{u.conv} to show that $u$ approximates $v$ in the set $S$ in the $C^p$ norm:
\begin{align*}
  \|u-v\|_{C^p(S)}&=\| u-w\|_{C^p(S)}\leq \sum_{i=1}^\infty\| h_i- w_i\|_{C^p(S)}\\
  &\leq  \sum_{i=1}^\infty\| h_i- \hw_i\|_{C^p(S)}+ \sum_{i=1}^\infty\| \hw_i- w_i\|_{C^p(S)}< 2\sum_{i=1}^\infty\frac\ep{2^{i+1}}=\ep\,,
\end{align*}
as we wanted to prove.
\end{proof}

\section{Noncompact level sets: the case of finite topological type}
\label{S.finite}

Our goal in the following two sections is to study which hypersurfaces can be level sets of solutions to the (Laplace or Yukawa) equation
\begin{equation*}
(\De-k^2)u=0
\end{equation*}
in $\RR^n$, with $k$ a real constant. For convenience, in this section we will focus on the case where the hypersurfaces have finite topological type, postponing the case of infinitely generated hypersurfaces until Section~\ref{S.infinite}.

In this section we will establish a sufficient condition that allows us to show that a wide class of hypersurfaces can be realized as level sets of a solution to the above equation modulo diffeomorphism. As mentioned in Section~\ref{S.guide}, a key step in the proof is to exploit the `freedom' associated to this diffeomorphism to embed the hypersurface in a way that permits us to have a fine control at infinity of various quantities. These suitably embedded hypersurfaces will be called tentacled.

It is standard that the notion of infinity in a hypersurface is captured in a precise way by its end structure. (Let us recall that, roughly speaking, an end of a noncompact manifold $L$ is a component of $L\minus K$ for a sufficiently large compact subset $K\subset L$; for the precise definition, cf.~\cite{HR96}). An end is said to be (smoothly) {\em collared} if it has a neighborhood diffeomorphic to $\Si×[0,\infty)$, where $\Si$ is a compact submanifold of codimension $1$ in $L$. Tentacled hypersurfaces (of finite type) are simply hypersurfaces with a finite number of ends, which are all collared and whose geometry is suitably controlled:

\begin{definition}\label{D.tentacled}
  An unbounded domain $\Om$ of $\RR^n$ with smooth boundary of finite type is {\em tentacled}\/ if one can find $J$ embedded images of $\RR^{n-1}$ in $\RR^n$, which we will call $\Pi_j$, such that the following statements hold (cf.\ Figure~\ref{F.tentacled}):
  \begin{enumerate}
  \item Each $\Pi_j$ divides $\RR^n$ into two domains, $H_j^+$ and $H_j^-$, and the sets $H_1^+,\dots, H_J^+$ are pairwise disjoint.
  \item The intersection $\Om\cap H_1^-\cap \cdots\cap H_J^-$ is bounded.
  \item Each connected component of the (possibly disconnected) hypersurface with boundary $\Om\cap(\Pi_1\cup\cdots \cup \Pi_J)$ is bounded and contained in an affine hyperplane. These components will be denoted by $\La_i$, with $1\leq i\leq N$. Here $N$ denotes the number of ends of the domain.
  \item Each `tentacle' $S_i$, which is the component of $\Om\cap (H_1^+\cup \cdots \cup H_J^+)$ connected with the `cap' $\La_i$, is isometric to the Riemannian product $\La_i× (0,\infty)$; in particular, the closure of $\La_i$ intersects orthogonally the boundary of $\Om$.
  \end{enumerate}

\begin{figure}[t]
  \centering
  {\psfrag{H1p}{$H_1^+$}\psfrag{H1m}{$H_1^-$}
    \psfrag{H2p}{$H_2^+$}\psfrag{H2m}{$H_2^-$}
    \psfrag{H3p}{$H_3^+$}\psfrag{H3m}{$H_3^-$} \psfrag{S1}{$S_1$}
    \psfrag{S2}{$S_2$} \psfrag{S3}{$S_3$} \psfrag{Pi1}{$\Pi_1$}
    \psfrag{Pi2}{$\Pi_2$} \psfrag{Pi3}{$\Pi_3$} \psfrag{La1}{$\La_1$}
    \psfrag{La2}{$\La_2$} \psfrag{La3}{$\La_3$} \psfrag{L}{$L$}
  \includegraphics[scale=0.2,angle=0]{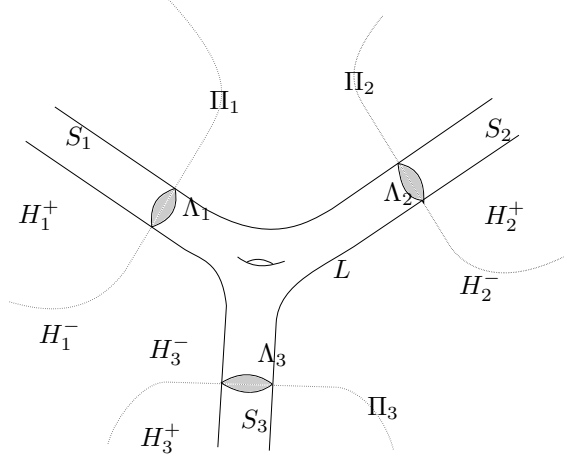}}
\caption{A tentacled surface $L$ diffeomorphic to the torus with $N=3$ ends.} 
\label{F.tentacled}
\end{figure}
  
  A (possibly disconnected) hypersurface of finite type is {\em tentacled} if it has a finite number of connected components, each of which is the boundary of a tentacled domain.
\end{definition}
\begin{remark}
  As we have allowed the intersection $\Om\cap \Pi_j$ to be disconnected, the different tentacles $S_i$ whose caps $\La_i$ are contained in the same intersection $\Om\cap \Pi_j$ can be linked among them. In this case, the number of ends $N$ is larger than the number of hyperplanes $J$.
\end{remark}

A very wide class of hypersurfaces of $\RR^n$ can be transformed into a tentacled hypersurface via an appropriate diffeomorphism of $\RR^n$. In the following proposition we will characterize the class of hypersurfaces that are equivalent to a tentacled hypersurface modulo diffeomorphism. Although the proof is elementary (and will be safely omitted), this result is of interest because it automatically provides a number of nontrivial examples of hypersurfaces diffeomorphic to a tentacled submanifold, which will be subsequently shown to be diffeomorphic to a level set of a harmonic function in $\RR^n$:

\begin{proposition}\label{P.algebraic}
  Let $M$ be a hypersurface of $\RR^n$ whose components are all
  noncompact and finite in number. Then there exists a diffeomorphism $\Phi$ of $\RR^n$ transforming $M$ into a tentacled hypersurface $\Phi(M)$ if and only if there is a compact hypersurface with boundary $\cM$ of the closure of the $n$-ball $\overline{\BB^n}$ such that:
  \begin{enumerate}
  \item The intersection of $\cM$ with the boundary of the $n$-ball is precisely its boundary $\pd\cM$, and this intersection is transverse.
    \item There is an embedding $h:\RR^n\to\overline{\BB^n}$ such that $h(\RR^n)=\BB^n$ and $h(M)$ is the interior of $\cM$.
    \end{enumerate}
\end{proposition}

\begin{example}\label{E.algebraic}
  By Richards' classification theorem~\cite{Ri63}, it is
  straightforward that any noncompact surface of finite type  (i.e.,
  each component is a genus $g$ torus with $N$ ends as considered in
  Question~\ref{Q.1}) can be embedded as a tentacled submanifold in
  $\RR^3$. More generally, a theorem of Calcut and King~\cite{CK04}
  ensures that any nonsingular real algebraic hypersurface of $\RR^n$ satisfies the hypotheses of Proposition~\ref{P.algebraic}, and can thus be realized as a tentacled hypersurface modulo diffeomorphism.
\end{example}

Let us begin by proving some intermediate lemmas. To construct the local solution of the equation $(\De-k^2)u=0$ that will be subsequently approximated by a global one, a basic tool is the Dirichlet Green's function $G_\Om(x,y)$ of the operator $\De-k^2$. Therefore, we will make use of a pointwise estimate  of the Green's function and its second-order derivatives that for our purposes is most conveniently stated as follows. Up to first order, this estimate is proved in Ref.~\cite[Theorem 5.7]{Ta97}; in the case of second-order derivatives, the argument follows the same lines and was kindly communicated to us by Yoichi Miyazaki (more generally, it yields $C^m$~estimates for the Green's function of appropriate elliptic operators of order $m$). For the benefit of the reader, we will sketch the proof of this result below. The definition of a uniform $C^r$ domain is given in~\cite[Definition 3.2]{Ta97}; in this paper, we will only need to know the (quite evident) fact that tentacled domains (or periodic tentacled domains, to be defined in Section~\ref{S.infinite} below) are uniform $C^r$ domains for all $r$.  

\begin{theorem}\label{T.Jap}
  Let $\Om$ be a uniform $C^4$ domain of $\RR^n$. Then, for any positive constant $\ep$ we have the pointwise estimate
  \[
\big|D_x^\al D_y^\be G_\Om(x,y)\big|\leq C|x-y|^{2-n-|\al|-|\be|}\e^{C'(\ep-(\la_\Om+k^2)^{1/2})|x-y|}
  \]
for the Green's function and its derivatives, which holds for all points $(x,y)\in\BOm×\BOm$ and any multiindices with $|\al|,|\be|\leq2$. Here $\la_\Om$ is the bottom of the spectrum of $-\De$ in the domain $\Om$ with Dirichlet boundary conditions and $C,C'$ are positive constants.
\end{theorem}
\begin{proof}
  When $|\al|,|\be|\leq 1$, this result is proved in~\cite[Theorem 5.7]{Ta97} using the key estimate~\cite[Theorem 5.5]{Ta97}
  \begin{equation}\label{LpLp}
\|D^\al(\De_\Om-\la)^{-1}\|_{L^p\to L^p}\leq C|\la_\Om+\la|^{\frac{|\al|}2-1}\,,
\end{equation}
valid for $|\al|\leq 2$. Here $\la$ is a possibly complex constant (whose real part we assume to be larger than $-\la_\Om$), $\|\cdot\|_{L^p\to L^q}$ stands for the operator norm $L^p(\Om)\to L^q(\Om)$ and we are denoting by $\De_\Om$ the Laplacian in $\Om$ with Dirichlet boundary conditions. Hence in what follows we will show how the argument can be modified to allow multiindices with $|\al|,|\be|\leq 2$.

Throughout this proof,  for the sake of simplicity we will denote by $C$ generic positive constants independent of $\la$. The $L^p\to L^q$ norm of the resolvent can be readily obtained from the above inequality and the Sobolev inequality
\begin{equation}\label{LpLq}
\|w\|_{L^q}\leq C \|w\|_{L^p}^{1-\frac n2(\frac1p-\frac1q)} \|w\|_{W^{2,p}}^{\frac n2(\frac1p-\frac1q)}\,,
\end{equation}
thereby finding that
\[
\|(\De_\Om-\la)^{-1}\|_{L^p\to L^q}\leq C |\la_\Om+\la|^{-1+\frac n2(\frac1p-\frac1q)}
\]
whenever $1<p<q<\infty$ and $p^{-1}-q^{-1}<2/n$.

It is a trivial matter to see that the $L^p\to L^p$ bound for the resolvent~\eqref{LpLp} yields the inequality $\|(\De_\Om-\la)^{-2}\|_{L^p\to L^p}\leq C|\la_\Om+\la|^{-2}$, while by elliptic regularity it is well known that $\|(\De_\Om-\la)^{-1}\|_{W^{1,p}\to W^{3,p}}\leq C$. Hence once obtains that
\[
\|(\De_\Om-\la)^{-2}\|_{L^p\to W^{3,p}}\leq  \|(\De_\Om-\la)^{-1}\|_{L^p\to W^{1,p}} \|(\De_\Om-\la)^{-1}\|_{W^{1,p}\to W^{3,p}} \leq C |\la_\Om+\la|^{-\frac12}\,.
\]
Therefore, using these estimates and the Sobolev inequality
\[
\|D^\al w\|_{L^\infty}\leq C\|w\|_{L^p}^{1-\frac{|\al|}3-\frac n{3p}} \|w\|_{W^{3,p}}^{\frac{|\al|}3+\frac n{3p}}
\]
for $|\al|\leq 2$ and $p>n$ we immediately get the $L^p\to L^\infty$ bound
\begin{equation}\label{LpLinfty}
\|D^\al(\De_\Om-\la)^{-2}\|_{L^p\to L^\infty}\leq C|\la_\Om+\la|^{-2+\frac12(|\al|+\frac np)}
\end{equation}
for the above range of parameters.

The next step is to prove that
\[
\|(\De_\Om-\la)^{-m}\|_{L^2\to L^p}\leq C|\la_\Om+\la|^{-m+\frac n2(\frac12-\frac 1p)}
\]
for all finite $p>n$ and large enough $m$. To this end, it suffices to take $m$ sufficiently large, so that there are numbers $2=:p_0<p_1<\cdots <p_m:=p$ whose consecutive inverses satisfy $p_{i-1}^{-1}-p_i^{-1}<\frac 2n$. Then one can apply the $L^p\to L^q$ estimate~\eqref{LpLq} to each pair of consecutive indices to yield the desired $L^2\to L^p$ bound.

Now we can combine the above equation and the estimate~\eqref{LpLinfty} to derive that, with $m:=j-2$,
\[
\|D^\al(\De_\Om-\la)^{-j}\|_{L^2\to L^\infty}\leq C|\la_\Om+\la|^{-j+\frac{|\al|}2+\frac n4}
\]
for $|\al|\leq 2$. Since the image of $(\De_\Om-\la)^{-j}$ is obviously contained in the Hölder space $C^{1-\frac np}$ provided $j$ is chosen as above, we can apply Tanabe's $ST^*$ lemma~\cite[Lemma 5.10]{Ta97} with $S:=D^\al(\De_\Om-\la)^{-j}$, $T:=D^\be(\De_\Om-\la)^{-j}$ and multiindices with $|\al|,|\be|\leq 2$ to derive that the integral kernel $K(x,y)$ of the operator $(\De_\Om-\la)^{-2j}$ is of class $C^2$ in $\Om×\Om$ and satisfies the pointwise bounds
\begin{align*}
  |D^\al_x D^\be_y K(x,y)|&\leq \|D^\al (\De_\Om-\la)^{-2j}D^\be\|_{L^1\to L^\infty}\\
  &\leq \|D^\al(\De_\Om-\la)^{-j}\|_{L^2\to L^\infty}\|(\De_\Om-\la)^{-j}D^\be\|_{L^1\to L^2}\\
  &\leq \|D^\al(\De_\Om-\la)^{-j}\|_{L^2\to L^\infty}\|D^\be(\De_\Om-\la)^{-j}\|_{L^2\to L^\infty}\\
  &\leq C|\la_\Om+\la|^{-2j+\frac12(n+|\al|+|\be|)}\,.
\end{align*}

This estimate for the kernel of $(\De_\Om-\la)^{-2j}$ readily yields the desired bounds for the Green's function. Indeed, the heat kernel $H(t,x,y)$ (that is, the integral kernel of $\e^{t\De_\Om}$) can be obtained from $K(x,y)$ as
\[
H(t,x,y)=\frac{(2j-1)!\,t^{1-2j}}{2\pi\I} \int_\Ga \e^{-t\la}K(x,y)\,d\la\,,
\]
where $\Ga$ is a contour enclosing the spectrum of $\De_\Om$. Hence the $C^2$ bound for $K(x,y)$ readily implies that
\[
|D^\al_x D^\be_yH(t,x,y)|\leq Ct^{-(n+|\al|+|\be|)/2}\,\e^{(\ep-\la_\Om) t}\,,
\]
where $\ep$ is an arbitrary positive constant. By Davies' perturbation method, this implies the Gaussian-type estimate
\[
|D^\al_x D^\be_yH(t,x,y)|\leq Ct^{-(n+|\al|+|\be|)/2}\,\e^{(\ep-\la_\Om) t}\,\e^{-C{|x-y|^2}/t}
\]
for $t>0$. Finally, the estimate for the Green's function $G_\Om(x,y)$ of the operator $\De_\Om-k^2$ follows by elementary methods from the heat kernel's Gaussian bound and the identity
\[
G_\Om(x,y)=\int_{0}^\infty\e^{-tk^2} H(t,x,y)\,dt\,.
\]
\end{proof}

In order to effectively apply the previous theorem to the study of tentacled hypersurfaces, we will need the lower bound for the eigenvalues that we shall prove in the following lemma, which will ensure the $C^2$ exponential decay at infinity of the Green's function of a tentacled domain:

\begin{lemma}\label{L.laOm}
  Let $\Om$ be either a tentacled domain in $\RR^n$ or a Riemannian product of the form $\La×(T,\infty)$ where $T\in[-\infty,\infty)$ and $\La$ is a bounded domain of $\RR^{n-1}$ with smooth boundary. Then $\la_\Om$ is strictly positive.
\end{lemma}
\begin{proof}
  It is a simple matter to show that $\la_{\La×(T,\infty)}$ equals the lowest Dirichlet eigenvalue~$\la_\La$ of the bounded domain $\La\subset\RR^{n-1}$, which is positive. Hence let us now assume that $\Om$ is a tentacled domain. Since~$0$ cannot be a Dirichlet eigenvalue of $\Om$, the result will follow once we show that~$0$ does not belong to the essential spectrum $\si_{\rm ess}(-\De_\Om)$ of the Laplacian on $\Om$ with Dirichlet boundary conditions. However, it is well known that $\si_{\rm ess}(-\De_\Om)= \si_{\rm ess}(-\De_{\Om\minus K})$ for any relatively compact set $K\subset\RR^n$ with smooth boundary. Therefore, taking $K:=\Om\cap H_1^-\cap\cdots \cap H_J^-$, in the notation of Definition~\ref{D.tentacled}, we find that
  \[
\si_{\rm ess}(-\De_\Om)=\si_{\rm ess}(-\De_{\bigcup_{i=1}^N S_i})= \bigcup_{i=1}^N\si_{\rm ess}(-\De_{S_i})\,.
\]
As the tentacle $S_i$ is isometric to the product $\La_i×(0,\infty)$, it follows from our first observation that each $\la_{S_i}$ is positive, thus completing the proof of the lemma.
\end{proof}

Before stating this section's main theorem, we need to establish some Green's function estimates for later use, the basic philosophy of which is to compare the Green's function $G_\Om(x,y)$ of a tentacled domain with that of a suitable domain with an Euclidean symmetry. To begin with, in the following lemma we will compare $G_\Om(x,y)$ with the Green's function $G_{S_i}(x,y)$ of the tentacle $S_i$ when the points $x$ and $y$ belong to this tentacle:

\begin{lemma}\label{L.GS}
  Let $\Om$ be a tentacled domain.  Then the pointwise estimate
  \begin{multline*}
     \big| D_x^\al D_y^\be G_\Om(x,y)-D_x^\al D_y^\be G_{S_i}(x,y)\big|\leq C_1 f_{\al\be}(x,y)\,
    \e^{-C_2[\dist(x,\La_i)+\dist(y,\La_i)]}
\end{multline*}
holds for all $|\al|,|\be|\leq 2$ whenever both points $x$ and $y$ lie in the tentacle $S_i$. Here $C_1$ and $C_2$ are positive constants and we have set
\begin{multline*}
f_{\al\be}(x,y):=\min\big\{\dist(x,\La_i)^{1-n-|\al|}\dist(y,\La_i)^{2-n-|\be|},\\ \dist(x,\La_i)^{2-n-|\al|}\dist(y,\La_i)^{1-n-|\be|}\big\}\,.
\end{multline*}
\end{lemma}
\begin{proof}
  Let us take two distinct points $x,y$ in the tentacle $ S_i$ and apply Green's identity to $G_{S_i}(\cdot,x)$ and $G_\Om(\cdot,y)$ in this tentacle to derive the expression
\begin{equation}\label{ident}
G_\Om(x,y)-G_{S_i}(x,y)=\int_{\La_i}G_\Om(z,y)\,\nu_i(z)\cdot \nabla_zG_{S_i}(z,x)\,d\si(z)\,.
\end{equation}
Here $d\si$ stands for the induced hypersurface measure on the `tentacle cap' $\La_i$, $\nu_i$ is the outer unit normal at $\La_i$ and we have used that $G_{S_i}(\cdot,x)=0$ on $\pd S_i$ and $G_\Om(\cdot,y)=0$ on $\pd S_i\minus\La_i$.

Taking derivatives with respect to $x$ and $y$ in the identity~\eqref{ident} and using Theorem~\ref{T.Jap}, one readily finds that
\begin{align*}
  \big| D_x^\al D_y^\be G_\Om(x,y)-D_x^\al D_y^\be & G_{S_i}(x,y)\big|\leq  \int_{\La_i}\big|D_y^\be G_\Om(z,y)\big|\, \big|\nabla_zD_x^\al G_{S_i}(z,x)\big|\,d\si(z)\\
  &\leq C_3\int_{\La_i} \frac{\e^{C_4[\ep-(\la_\Om+k^2)^{1/2}]|z-y|+C_4[\ep-(\la_{S_i}+k^2)^{1/2}]|z-x|}}{|z-y|^{n+|\be|-2}|z-x|^{n+|\al|-1}} \,d\si(z)\\
  &\leq \frac{C_1\e^{C_4[\ep-(\la_\Om+k^2)^{1/2}]\dist(y,\La_i)+C_4[\ep-(\la_{S_i}+k^2)^{1/2}]\dist(x,\La_i)}}{\dist(y,\La_i)^{n+|\be|-2}\dist(x,\La_i)^{n+|\al|-1}}
\end{align*}
for $|\al|,|\be|\leq 2$. As $C_2:=C_4 \min\{(\la_\Om+k^2)^{1/2}-\ep,(\la_{S_i}+k^2)^{1/2}-\ep\}$ can be taken positive even if $k=0$ by Lemma~\ref{L.laOm}, the above inequality proves the lemma with
\[
f_{\al\be}(x,y)=\dist(x,\La_i)^{1-n-|\al|} \dist(y,\La_i)^{2-n-|\be|}\,.
\]
To show that this estimate also holds with
\[
f_{\al\be}(x,y)=\dist(x,\La_i)^{2-n-|\al|} \dist(y,\La_i)^{1-n-|\be|}\,,
\]
thereby completing the proof of the lemma, it suffices to exchange $x$ and $y$ in Eq.~\eqref{ident} by the symmetry of the Green's functions and repeat the argument.
\end{proof}

In the following lemma we will prove the exponential decay of the Green's function $G_\Om(x,y)$ when the points $x$ and $y$ lie in distinct tentacles:

\begin{lemma}\label{L.GOmS}
  Let $x\in S_i$ and $y\in S_j$ be points lying in distinct tentacles of a tentacled domain $\Om$. Then the Green's function $G_\Om(x,y)$ decays as
  \[
\big|D_x^\al D_y^\be G_\Om(x,y)\big|\leq C_1\dist(x,\La_i)^{1-n-|\al|}\dist(y,\La_i)^{2-n-|\be|} \e^{-C_2 [\dist(x,\La_i)+\dist(y,\La_i)]}
  \]
  for some positive constants $C_1,C_2$ and all $|\al|,|\be|\leq2$ .
\end{lemma}
\begin{proof}
  Given $x\in S_i$ and $y\in S_j$, one can apply Green's identity to $G_{S_i}(\cdot,x)$ and $G_\Om(\cdot,y)$ in $S_i$ to obtain the formula
\[
G_\Om(x,y)=\int_{\La_i} G_\Om(z,y)\, \nu_i(z)\cdot \nabla_zG_{S_i}(z,x)\, d\si(z)\,.
\]
Taking derivatives with respect to $x$ and $y$ and using Theorem~\ref{T.Jap}, an argument as in Lemma~\ref{L.GS} yields
\begin{align*}
  \big| D_x^\al D_y^\be &G_\Om(x,y)\big|\leq  \int_{\La_i}\big|D_y^\be G_\Om(z,y)\big|\, \big|\nabla_zD_x^\al G_{S_i}(z,x)\big|\,d\si(z)\\
  &\leq C_3\int_{\La_i} \frac{\e^{C_4[\ep-(\la_\Om+k^2)^{1/2}]|z-y|+C_4[\ep-(\la_{S_i}+k^2)^{1/2}]|z-x|}}{|z-y|^{n+|\be|-2}|z-x|^{n+|\al|-1}} \,d\si(z)\\
  &\leq \frac{C_1\e^{C_4[\ep-(\la_\Om+k^2)^{1/2}\dist(y,\La_i)]+C_4[\ep-(\la_{S_i}+k^2)^{1/2}]\dist(x,\La_i)}}{ \dist(x,\La_i)^{n+|\al|-1}\dist(y,\La_i)^{n+|\be|-2}}\,.
\end{align*}
  The claim then follows by noticing that $C_2:=C_4 \min\{(\la_\Om+k^2)^{1/2}-\ep,(\la_{S_i}+k^2)^{1/2}-\ep\}$ can be taken positive even for $k=0$ by Lemma~\ref{L.laOm}.
\end{proof}

We shall next prove the main result of this section, namely, that any tentacled hypersurface can be transformed by a small diffeomorphism into a level set of a global solution to the equation $(\De-k^2)u=0$. The proof follows the strategy we outlined in Section~\ref{S.guide}, using the above estimates for the Green's functions to ensure that one can define a local solution of the equation that has a level set diffeomorphic to the tentacled hypersurface and satisfies the hypotheses of the $C^1$ stability theorem:

\begin{theorem}\label{T.finite}
  Let $L\subset\RR^n$ be a (possibly disconnected) tentacled hypersurface of finite type and let $k$ be a real constant. Then one can transform the hypersurface $L$ by a diffeomorphism $\Phi$ of $\RR^n$, as close to the identity as we wish in the $C^1$ norm, so that $\Phi(L)$ is a union of connected components of a level set $u^{-1}(c_0)$ of a solution of the equation $(\De-k^2)u=0$ in $\RR^n$.
\end{theorem}
\begin{proof}
  We start by showing that, given any connected component $L_0$ of the hypersurface $L$, there exists a local solution of the equation having a level set diffeomorphic to this component. To this end, let us denote by $\Om$ the tentacled domain whose boundary is $L_0$. We will keep the notation $S_i$ (with $1\leq i\leq N$) for the tentacles of the domain $\Om$ (as in Definition~\ref{D.tentacled}), which can be characterized as
  \begin{equation}\label{Si}
S_i:=\big\{y+t\nu_i:y\in \La_i,\; t>0\big\}\,.
\end{equation}
Here the constant vector $\nu_i$ is the outer unit normal at the tentacle cap $\La_i$.
  
The construction of the desired local solution will make use of the Green's function $G_\Om(x,y)$ of the domain $\Om$. To ensure that the local solution satisfies the hypotheses of the $C^1$ noncompact stability theorem, it is convenient to start by considering a straight half-line $\ga_i$ in each tentacle. That is, we fix a point $y_i$ in each tentacle cap $\La_i$ and define the half-line $\ga_i$ as
\begin{equation}\label{gai}
\ga_i:=\big\{y_i+t\nu_i:t>1\big\}\,.
\end{equation}
The length measure on $\ga_i$ will be denoted by $\mu_i$. A sketch of
many of the geometric objects that appear in the proof of this theorem is given in
Figure~\ref{F.proof}.

\begin{figure}[t]
  \centering
  {\psfrag{Si}{$S_i$}
    \psfrag{Ci}{$\ga_i$} \psfrag{yi}{$y_i$} \psfrag{Lai}{$\La_i$}
    \psfrag{V0}{$V$} \psfrag{vc0}{$v^{-1}(c_0)$} \psfrag{U}{$U$}
  \includegraphics[scale=0.2,angle=0]{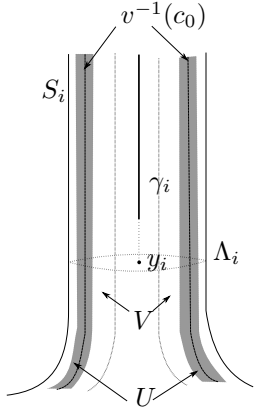}}
\caption{Definition of the half-line $\ga_i$ and the neighborhoods $V$
  and $U$ corresponding to a tentacle $S_i$.} 
\label{F.proof}
\end{figure}

The local solution will be constructed later on using the positive function
\begin{equation*}
v_i(x):=\int_{\ga_i}G_\Om(x,y)\,d\mu_i(y)\,.
\end{equation*}
As a consequence of the estimates for the Green's function we proved in Theorem~\ref{T.Jap} and Lemma~\ref{L.laOm}, one can readily check that the function $v_i$ satisfies
\begin{align}
  v_i(x)&= \int_1^\infty G_\Om(x,y_i+t\nu_i)\,dt\leq C'_1\int_1^\infty |x-y_i-t\nu_i|^{2-n}\e^{-C'_2|x-y_i-t\nu_i|}\,dt\notag \\
  &\leq C'_1 \int_{-\infty}^\infty \big(\dist(x,\ga_i)^2+\tau^2\big)^{1-\frac n2}\e^{-C'_2(\dist(x,\ga_i)^2+\tau^2)^{1/2}}\,d\tau\notag \\
  &\leq 2C'_1\int_{\dist(x,\ga_i)}^\infty s^{2-n}\e^{-C'_2s}\,ds\notag \\
  &\leq  \begin{cases}
    C'_3\dist(x,\ga_i)^{3-n}\e^{-C'_2\dist(x,\ga_i)}&\text{if }n\geq4,\\
    C'_3(1+|\log\dist(x,\ga_i)|)\,\e^{-C'_2\dist(x,\ga_i)/2} &\text{if }n=3
  \end{cases}\label{vi}
\end{align}
for some positive constants $C'_j$. In the above inequalities, the variables $t$ and $\tau$ are related by $\tau:=t-|y_i-Y_i(x)|$, where $Y_i(x)$ is the unique point of $\ga_i$ such that $|x-Y_i(x)|=\dist(x,\ga_i)$, and $s:=(\dist(x,\ga_i)^2+\tau^2)^{1/2}$. By the above estimates,  $v_i$ is well defined; indeed, it can be readily checked that it is of class $C^\infty(\BOm\minus\ga_i)\cap L^1\loc(\Om)$ and satisfies the equation
\begin{equation*}
  (k^2-\De)v_i=\mu_i
\end{equation*}
and the boundary condition $v_i|_{L_0}=0$.

Our desired local solution will be the sum 
\[
v:=\sum_{i=1}^Nv_i\,,
\]
which is smooth in the closure of the domain $\Om$ minus the union of all the half-lines $\ga_i$ and satisfies the equation
\[
(k^2-\De)v=\sum_{i=1}^N\mu_i
\]
with boundary condition $v|_{L_0}=0$. Our next goal is to show that the function $v$ satisfies the saturation, gradient and $C^2$-boundedness conditions of the $C^1$ noncompact stability theorem. That is, we claim that there exists a half-neighborhood $V\subset \Om$ of the component $L_0$ and a positive constant $\eta$ such that the function
$v$ satisfies the gradient condition
\begin{equation}\label{gradv}
  |\nabla v|\geq C>0
\end{equation}
in a set $V$ that is saturated by $v$ in the sense that any component of $v^{-1}(c)$ connected with $V$ is contained in $V$ for all $0<c<\eta$, and that moreover the second-order derivatives of the function $v$ are bounded in $V$.

In order to prove this claim, we introduce the auxiliary function
\begin{equation}\label{tildevi}
\tv_i(x):=\int_{\tilde\ga_i}G_{\tS_i}(x,y)\,d\tilde\mu_i(y)
\end{equation}
which is of class $C^\infty(\tS_i\minus\tilde\ga_i)\cap L^1\loc(\tS_i)$ by the same argument leading to the estimate~\eqref{vi}. Here we are respectively denoting by
\[
\tS_i:=\big\{y+t\nu_i:y\in \La_i,\; t\in\RR\big\}\quad\text{and}\quad \tilde\ga_i:=\big\{y_i+t\nu_i: t\in\RR\big\}
\]
the cylinder and straight line corresponding to the tentacle $S_i$ and to the half-line $\ga_i$, while $\tilde\mu_i$ stands for the length measure on the line $\tilde\ga_i$.

A simple symmetry argument shows that the value of the function $\tv_i$ at an arbitrary point of the cylinder $\tS_i$ can be expressed in terms of the Green's function $G_{\La_i}$ of the tentacle cap via
\begin{equation}\label{twG1}
\tv_i(y+t\nu_i)=G_{\La_i}(y,y_i)\,,
\end{equation}
where we are parametrizing the points in the cylinder as $y+t\nu_i$ as we did in Eq.~\eqref{Si} and we naturally identify the cap $\La_i$ with a bounded domain of $\RR^{n-1}$ using the property~(iii) of Definition~\ref{D.tentacled}. The normal derivative of the function $\tv_i$ at the boundary of the cylinder $\tS_i$ can similarly computed using the symmetry as
\begin{equation}\label{twG2}
\pd_\nu\tv_i(y+t\nu_i)=\pd_\nu G_{\La_i}(y,y_i)\,,
\end{equation}
where $y$ is any point in the boundary of the cap $\La_i$ and $t\in\RR$.

By Hopf's boundary point lemma~\cite{GT98}, it follows that the above normal derivative $\pd_\nu G_{\La_i}(\cdot,y_i)$ is strictly negative, so the boundedness of the cap $\La_i$ allows us to infer that 
\[
\big|\nabla G_{\La_i}(\cdot,y_i)\big|\geq C_1>0
\]
on a half-neighborhood $U_i\subset\La_i$ of the boundary $\pd\La_i$. Using again the fact that the cap $\La_i$ is bounded, it is standard that this set $U_i$ can be safely assumed to be saturated by the function $G_{\La_i}(\cdot,y_i)$, meaning that for any $c$ at most one connected component of the level sets $\{y\in\La_i:G_{\La_i}(y,y_i)=c\}$ intersects $U_i$ and that this component is actually contained in $U_i$. By the symmetry conditions~\eqref{twG1} and \eqref{twG2}, this ensures that there exists a positive constant $\eta_1$ such that the gradient condition
\begin{equation}\label{tvC}
\big|\nabla \tv_i\big|\geq C_2>0
\end{equation}
holds in the set $\tilde V_i:=\{y+t\nu_i: y\in U_i,\; t\in\RR\}$. It should be noticed that all the derivatives of $\tv_i$ are obviously bounded in $\tilde V_i$ by symmetry. By the definition of the half-neighborhoods $U_i$, the set $\tilde V_i$ contains a unique component of the level set $\tv_i^{-1}(c)$ for all values $0<c<\eta_1$. Therefore, the above discussion shows that the auxiliary function $\tv_i$ satisfies the requirements of the $C^1$ noncompact stability theorem.

Motivated by this, our next step towards proving that the solution $v$ also satisfies the conditions of the stability theorem is to control the difference between the local solution $v$ and the auxiliary function $\tv_i$ in the tentacle $S_i$. To this end, let us take a point $x\in S_i$ and estimate this difference as
\begin{align}
  \big|D^\al v(x)&- D^\al\tv_i(x)\big|\leq \int_{\ga_i}\big|D_x^\al G_\Om(x,y)-D_x^\al G_{S_i}(x,y)\big|\,d\mu_i(y)\notag\\
  & +\int_{\ga_i}\big|D_x^\al G_{\tS_i}(x,y)-D_x^\al G_{S_i}(x,y)\big|\,d\mu_i(y)\notag\\
  &+\sum_{1\leq j\neq i\leq N} \int_{\ga_j}\big|D_x^\al G_\Om(x,y)\big|\,d\mu_j(y)+ \int_{\tilde\ga_i\minus\ga_i}\big|D_x^\al G_{\tS_i}(x,y)\big|\,d\tilde\mu_i(y)\,.\label{roll}
\end{align}
Since $\tS_i$ is  a tentacled domain itself, one can now apply Lemmas~\ref{L.GS} and~\ref{L.GOmS} to obtain that for $|\al|\leq 2$ the first three integrals can be upper bounded by the exponential $C_3\,\e^{-C_4\dist(x,\La_i)}$, where $C_j$ are positive constants. To estimate the last integral, let us denote by $\tilde y_i:=y_i+\nu_i$ the endpoint of the half-line $\ga_i$ and apply Theorem~\ref{T.Jap} and Lemma~\ref{L.laOm} to derive that
\begin{align*}
  \int_{\tilde\ga_i\minus\ga_i}\big|D_x^\al G_{\tS_i}(x,y)\big|\,d\tilde\mu_i(y)&\leq \int_{\tilde\ga_i\minus\ga_i} \frac{C_5\,\e^{-C_6|x-y|}}{|x-y|^{n+|\al|-2}}\,d\tilde\mu_i(y)\leq C_7\,\e^{-C_8|x-\tilde y_i|}
\end{align*}
whenever the distance from the point $x$ to the cap $\La_i$ is greater than $2$ and $|\al|\leq2$. Hence we obtain the pointwise $C^2$ estimate
\begin{equation}\label{exp}
\max_{|\al|\leq 2}\big|D^\al v(x)- D^\al\tv_i(x)\big|\leq C_9\,\e^{-C_{10}|x|}\,,
\end{equation}
which holds in the tentacle $S_i$ provided that $|x|$ is large enough.

Armed with these preliminary results, we can prove that the local solution $v$ satisfies the hypotheses of the $C^1$ noncompact stability theorem. As a first observation, notice that, the domain $\Om$ having a finite number of ends $N$, the gradient bound~\eqref{tvC} and the $C^2$ estimate~\eqref{exp} imply that there is a positive constant $\eta_2$ and a compact subset $K$ of $\RR^n$ such that the gradient of the local solution satisfies $|\nabla v|\geq C'>0$ in the set
\[
W:=\bigg(v^{-1}\big((0,\eta_2)\big) \cap\bigcup_{i=1}^N\tilde V_i\bigg)\minus K
\]
and that the second-order derivatives of $v$ are bounded in $W$. (We recall each set $\tilde V_i$ was chosen so that the auxiliary function $\tv_i$ and its derivatives satisfied appropriate bounds in it.) We can safely assume that, for any connected component $W_0$ of $W$, there is a unique component of $v^{-1}(c)\minus K$ meeting $W_0$ for all values $c\in(0,\eta_2)$ and that this latter component of the level set does not intersect the boundary of $W_0$ but at the compact set $\pd K$.

The set $W$ should be thought of as a conveniently chosen half-neighborhood of the hypersurface $L_0$ minus a compact set. As the local solution $v$ satisfies suitable gradient and saturation conditions in $W$ by the above argument, now it essentially suffices to deal with $v$ in the compact set $K$. Indeed, by Hopf's boundary point lemma~\cite{GT98} there are positive constants $C_2$ and $\eta_3$ and a half-neighborhood $W'\subset\Om$ of the intersection $L_0\cap K$ where the gradient of the local solution is bounded as $|\nabla v|\geq C_2$. Moreover, by compactness it is obvious that the set $W'$ can be chosen so that for all values $c\in (0,\eta_3)$ there is a unique component of $v^{-1}(c)\cap K$ that meets the set $W'$, this component intersecting the boundary $\pd W'$ only on $\pd K$.

Putting together these results, Eq.~\eqref{gradv}, which ensures that the local solution $v$ satisfies the gradient and saturation conditions of Theorem~\ref{T.stab}, now follows by taking the constant $\eta:=\min\{\eta_2,\eta_3\}$ and choosing an appropriate subset $V\subset W\cup W'$. It also stems that the norm $\|v\|_{C^2(V)}$ is finite.

Before we can profitably apply the stability theorem to the local solution $v$, there is one last technical point we must take care of. The equation $(\De-k^2)v=0$ is only satisfied in the half-neighborhood $V$ of the hypersurface $L_0$, not in its closure. Therefore, in order to apply the theorem we will first prove that there is a level set of $v$ in $V$ diffeomorphic to $L_0$ via a $C^1$-small diffeomorphism (e.g., via a diffeomorphism $\Psi_0$ with $\|\Psi_0-\id\|_{C^1(\RR^n)}<\ep/2$) which only differs from the identity in a neighborhood of $\overline{V}$. This is easily shown by taking an open set $V'$ containing the closure of $V$ and a smooth extension $\bar v$ of the local solution $v$ to the set $(V'\cup\Om)\minus\ga$ which is equal to $v$ in $\Om\minus\ga$ and negative in ${V'\minus\BOm}$ (but does not necessarily satisfy the equation $(\De-k^2)\bar v=0$). One can then apply Corollary~\ref{C.Thom} with $(f,g,U,V,L)=(\bar v,\bar v-c,V'\cup(\Om\minus\bigcup_{i=1}^N\ga_i),V,L_0)$ and $p=1$ to deduce the result, where $c$ is a small enough constant in the interval $(0,\eta)$.

Applying the same reasoning to all the connected components of the hypersurface $L$, we derive that there exist a diffeomorphism $\Psi$ of $\RR^n$ with $\|\Psi-\id\|_{C^1(\RR^n)}<\ep/2$ and a function $w$, which satisfies the equation $(\De-k^2)w=0$ in the closure of a neighborhood $U$ of $\Psi(L)$, such that:
\begin{enumerate}
\item The transformed hypersurface $\Psi(L)$ is a level set $w^{-1}(c_0)$ of the function, for some positive $c_0$.
\item The neighborhood $U$ is saturated, that is, if the intersection of $w^{-1}(c)\cap U$ is nonempty for some $c\in\RR$, then $w^{-1}(c)$ does not intersect the boundary~$\pd U$.
\item The local solution satisfies the gradient condition $|\nabla w|\geq C'>0$ in $U$ and its second-order derivatives are bounded in this set.
\item The complement of the set $U$ in $\RR^n$ does not have any compact components.
\end{enumerate}

To complete the proof of the theorem, it suffices to approximate the local solution $w$ in the set $U$ by a global solution $u$ of the equation $(\De-k^2)u=0$. By the condition~(iv) above, one can invoke Theorem~\ref{T.approx} to do so, ensuring that the $C^2$ norm $\|u-w\|_{C^2(U)}$ is arbitrarily small. If this norm is chosen small enough, the construction of the local solution $w$ and the saturated set $U$ allows us to apply Theorem~\ref{T.stab} to each connected component of the level set $w^{-1}(c_0)$ to obtain $C^1$-small diffeomorphisms that are only different from the identity in a prescribed neighborhood of the component and transform each component of $w^{-1}(c_0)$ into components of a level set of the global solution $u$. As there is no loss of generality in assuming that the supports of these diffeomorphisms minus the identity are pairwise disjoint, we therefore obtain a diffeomorphism $\hPsi$ of $\RR^n$ with $\|\hPsi-\id\|_{C^1(\RR^n)}$ as small as we wish (say, smaller than $\ep/2$) transforming the level set $w^{-1}(c_0)$ into a union of components of a level set of $u$. The diffeomorphism $\Phi:=\hPsi\circ\Psi$ then transforms the hypersurface $L$ into a union of components of $u^{-1}(c_0)$ and is arbitrarily close to the identity in the sense that $\|\Phi-\id\|_{C^1(\RR^n)}<\ep$.
\end{proof}

\begin{remark}\label{R.solomonic}
  The method of proof remains valid if we do not demand the ends of the tentacled domains to be `straight' (i.e., isometric to $\La_i× (0,\infty)$) but 'of solomonic column type' (i.e., isometric to the intersection of $\La_i× (0,\infty)$ with a domain invariant under a free isometric $\ZZ$-action). 
\end{remark}

\section{Noncompact level sets: the case of infinite topological type}
\label{S.infinite}

In this section we will conclude the proof of Theorem~\ref{T.1} by considering the case of hypersurfaces that are not finitely generated. Although the basic philosophy of the proof is the same as in Theorem~\ref{T.finite}, in this case the hypersurfaces under consideration typically have an infinite number of ends (so, in particular, they are not diffeomorphic to an algebraic variety) that are not necessarily collared, and this introduces additional difficulties that require a separate treatment. The simplest example of a hypersurface of this kind is the torus of infinite genus in $\RR^3$ (cf.\ Figure~\ref{F.jungle}). This example shows that it is very convenient to embed the hypersurfaces so as to exploit discrete translational symmetries, so we will start by introducing some notation associated to these symmetry groups.

Let us fix a positive integer $\ell$ not greater than the space dimension $n$. We take a set of linearly independent vectors $\cA:=\{a_1,\dots,a_\ell\}\subset\RR^n$ and denote by $a_j^*$ their dual vectors, which are the only elements in the linear span of the vectors $\cA$ satisfying $ a_i\cdot a^*_j=\de_{ij}$. For each $t\in\ZZ^\ell$ we will then denote by $\tau_t^\cA:\RR^n\to\RR^n$ the map
\[
\tau_t^\cA(x):=x+t_1a_1+\cdots+t_\ell a_\ell\,,
\]
which defines a free isometric $\ZZ^\ell$-action. We will also consider the fundamental cell
\begin{equation*}
Q^\cA:=\big\{s_1a_1+\cdots+s_\ell a_\ell+b:0<s_i<1,\; b\in\RR^n\;\text{and } b\cdot a_i=0 \;\text{for  }1\leq i\leq\ell \big\}
\end{equation*}
associated to this action and the faces
\[
\Pi_j^\cA:=\big\{x\in\pd{Q^\cA}: x\cdot a_j^*=0\big\}\,,
\]
with $1\leq j\leq \ell$.

We will say a set $\cU$ of $\RR^n$ is {\em $\cA$-periodic} if it is invariant under the above $\ZZ^\ell$ action, i.e., if $\tau_t^\cA(\cU)=\cU$ for all $t\in\ZZ^\ell$. If a set $\cU$ is $\cA$-periodic, $\cU$ can be recovered from its intersection with the fundamental cell via the identity
\begin{equation}\label{Omt}
\cU=\bigcup_{t\in\ZZ^\ell}\tau^\cA_t(\cU\cap \overline{Q^\cA})\,.
\end{equation}
For simplicity, we shall sometimes say that a set is {\em $\ell$-periodic} if it is $\cA$-periodic for some set $\cA$ with $\ell$ independent vectors.

Basically, the motivation of this section is to prove that there are
global solutions to the equation $(\De-k^2)u=0$ having a level set
diffeomorphic to infinite connected sums of any nonsingular algebraic hypersurface. From the experience of Theorem~\ref{T.finite} one can guess that it will be useful to exploit this diffeomorphism to embed the infinite-type hypersurface (in this case, the aforementioned infinite sum) so that both the collared ends of the underlying algebraic hypersurface and the way the different individual hypersurfaces are glued together are `geometrically controlled' at infinity. We will do this through the following definition:

\begin{definition}\label{D.infinite}
  An  $\cA$-periodic  domain $\cU$ of $\RR^n$ with smooth boundary is
  {\em tentacled} if its intersection with the fundamental cell
  $Q^\cA$ is either relatively compact or equal to  a tentacled domain
  of finite type $\Om$ minus a compact subset $K$ of $\RR^n$. A {\em tentacled hypersurface} of $\RR^n$ of possibly infinite type is a hypersurface with a finite number of connected components, each of which is the boundary of a (possibly periodic) tentacled domain.
\end{definition}
\begin{remark}
  A tentacled hypersurface of infinite type does not need to be periodic, even if all its components are. Moreover, the periodic components can have distinct symmetry groups.
\end{remark}

It should be noted that if the intersection of the periodic tentacled domain $\cU$ with the fundamental cell $Q^\cA$ is unbounded, obviously the rank $\ell$ of the symmetry group is at most $n-1$.

The class of tentacled hypersurfaces of infinite type modulo diffeomorphism includes infinite connected sums of a large class of hypersurfaces, as we will see in the following examples:

\begin{example}
  If $L$ is a (possibly compact) nonsingular algebraic hypersurface of $\RR^n$, there is an $\ell$-periodic tentacled hypersurface that is diffeomorphic to a connected sum of infinitely many copies of $L$. Here the rank $\ell$ can take any value between $1$ and $n-1$ (resp.\ $n$) if $L$ is noncompact (resp.\ compact). In particular, and getting back to Question~\ref{Q.1}, the torus of infinite genus and the infinite jungle gym are examples of $1$-periodic and $3$-periodic tentacled surfaces of $\RR^3$, respectively.
\end{example}

\begin{example}
  Given any integer $\ell$ between $1$ and $n-1$ and a tentacled hypersurface $L$ in $\RR^n$, there is an $\ell$-periodic hypersurface that is diffeomorphic to a  connected sum of infinitely many copies of $L$. This readily follows from the following elementary proposition, which is a trivial consequence of the fact that any tentacled hypersurface is collared:

\begin{proposition}\label{P.plane}
  Given a tentacled hypersurface $L\subset\RR^n$ and a set $\cA\subset\RR^n$ of $\ell$ independent vectors ($\ell\leq n-1$), one can transform it by a diffeomorphism $\Psi$ of $\RR^n$ so that $\Psi(L)$ is tentacled and contained in the fundamental cell $Q^\cA$.
\end{proposition}
\end{example}

As in the previous section (cf.\ Lemma~\ref{L.laOm}), firstly we need
to prove that the spectrum of the Dirichlet Laplacian in a periodic
tentacled domain is bounded away from $0$ in order to obtain
exponential decay of the Green's function. This is done in the
following lemma, which exploits both the asymptotic Euclidean
symmetries of tentacled domains of finite type and the invariance under the isometric $\ZZ^\ell$ action of periodic tentacled domains:

\begin{lemma}\label{L.periodic}
  Let $\cU$ be an $\cA$-periodic tentacled domain in $\RR^n$. Then the bottom of the spectrum of the Laplacian $\la_\cU$ in this domain is  positive.
\end{lemma}
\begin{proof}
Let us begin by observing that the Dirichlet spectrum in $\cU$ can be written as
\begin{equation}\label{Detheta}
\si(\De_\cU)=\bigcup_{\theta\in\TT^\ell} \si(\De_\theta)\,,
\end{equation}
where $\si(\De_\theta)$ denotes the $L^2$ spectrum of the Laplacian in $\cU\cap Q^\cA$ with the boundary conditions
\begin{align}
&w=0\quad\text{on }\pd\cU\cap\overline{Q^\cA}\,,\label{BCtheta}\\
&D^\al w(x+a_j)=\e^{\I \theta_j} D^\al w(x)\quad\text{for all }x\in\cU\cap\Pi^\cA_j,\; |\al|\leq1,\; 1\leq j\leq \ell.\notag
\end{align}
Here $\TT^\ell:=(\RR/2\pi\ZZ)^\ell$ is the $\ell$-torus and we write $\theta=(\theta_1,\dots,\theta_\ell)$. As for Eq.~\eqref{Detheta}, the inclusion of $\si(\De_\cU)$ in the union of the spectra of $\De_\theta$ follows from a standard modification of Floquet theory~\cite{Su88}, while the fact that both sets actually coincide follows e.g.\ from Adachi's results on unitary actions of amenable groups~\cite{Ad95}.

From the boundary condition~\eqref{BCtheta} and the fact that the boundary $\pd\cU$ intersects the fundamental cell ${Q^\cA}$, it follows that $0$ cannot be an eigenvalue of $\De_\theta$ for any $\theta\in\TT^\ell$. Besides, $0$ cannot belong to the essential spectrum of $\De_\theta$ either, since the spectrum $\si(\De_{\cU\cap Q^\cA})$ is bounded away from $0$ and $\si_{\rm ess}(\De_\theta)$ coincides with the Dirichlet essential spectrum $\si_{\rm ess}(\De_{\cU\cap Q^\cA})$ by the boundedness of the intersection $\cU\cap\Pi^\cA_j$. The former assertion is clear when $\cU\cap Q^\cA$ is bounded and stems from Lemma~\ref{L.laOm} when $\cU\cap Q^\cA$ is a tentacled hypersurface minus a compact set. The statement now follows from the decomposition~\eqref{Detheta} and the compactness of $\TT^\ell$.
\end{proof}

We shall next prove the main result of this section where we adapt the
method of proof of Theorem~\ref{T.finite} to construct solutions of
the equation $(\De-k^2)u=0$ with a level set diffeomorphic to any
tentacled hypersurface of infinite type. To avoid unnecessary repetitions, we will not present in full detail some steps in the argument, referring instead to the appropriate parts of the demonstration of Theorem~\ref{T.finite}. Together with Theorem~\ref{T.finite}, this completes the proof of Theorem~\ref{T.1}.

\begin{theorem}\label{T.infinite}
  Let $k$ be a real constant. Given a (possibly disconnected and of infinite type) tentacled hypersurface $L\subset\RR^n$, we can transform it by a diffeomorphism $\Phi$ of $\RR^n$, arbitrarily close to the identity in the $C^1$ norm, so that $\Phi(L)$ is a union of connected components of a level set $u^{-1}(c_0)$ of a solution to the equation $(\De-k^2)u=0$ in~$\RR^n$.
\end{theorem}
\begin{proof}
As in the proof of Theorem~\ref{T.finite}, our goal is to construct a local solution of the equation defined in a half-neighborhood of each component $L_0$ of the hypersurface $L$ and satisfying the saturation, gradient and $C^2$-boundedness conditions of the $C^1$ noncompact stability theorem. If the component $L_0$ is a tentacled hypersurface of finite type, this local solution was constructed in the proof of Theorem~\ref{T.finite}, so we will assume that $L_0$ is $\cA$-periodic for a set of $\ell$ linearly independent vectors $\cA$. We will denote by $\cU$ the $\cA$-periodic tentacled domain enclosed by $L_0$; the main geometric objects considered in this proof are presented in Figure~\ref{F.periodic}.

\begin{figure}[t]
  \centering
  {\psfrag{QA}{$Q^\cA$}
    \psfrag{ga1}{$\ga_1$} \psfrag{ga2}{$\ga_2$} \psfrag{ga3}{$\ga_3$}
    \psfrag{V}{$V$} \psfrag{cU}{$\cU$} \psfrag{dots}{\dots}
  \includegraphics[scale=0.2,angle=0]{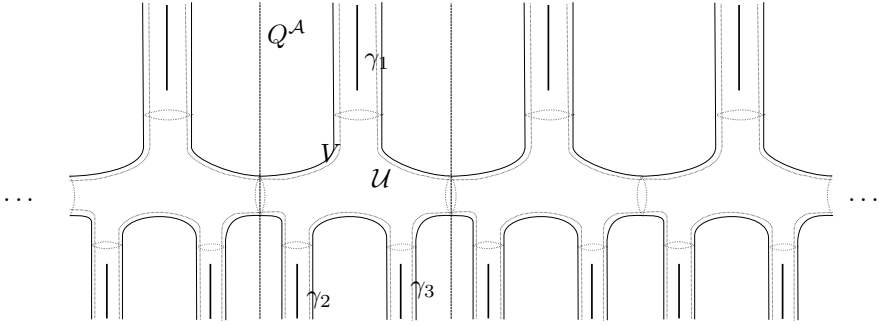}}
\caption{The half-lines $\ga_i$, the fundamental cell $Q^\cA$ and the neighborhood $V$ that appear in the proof of Theorem~\ref{T.infinite}.} 
\label{F.periodic}
\end{figure}

By the definition of an $\cA$-periodic tentacled domain, the intersection $\cU\cap Q^\cA$ of this domain with the fundamental cell is either relatively compact or equal to $\Om\minus K$, with $\Om$ a tentacled domain with $N$ ends and $K$ a compact set. Let us first suppose that $\cU\cap Q^\cA=\Om\minus K$. In this case, we can safely assume that the tentacles $S_i$ of the domain $\Om$ do not intersect the compact set $K$, and define the half-lines $\ga_i\subset S_i$ as we did in Eq.~\eqref{gai}. Let us set
\[
v_i(x):=\int_{\ga_i} G_\cU(x,y)\,d\mu_i(y)\,,
\]
with $\mu_i$ being the length measure on the half-line $\ga_i$. By Lemma~\ref{L.periodic} and the exponential bound for the Green's function proved in Theorem~\ref{T.Jap}, one immediately obtains the estimate
\[
v_i(x)\leq\begin{cases}
    C_1\dist(x,\ga_i)^{3-n}\e^{-C_2\dist(x,\ga_i)}&\text{if }n\geq4,\\
    C_1(1+|\log\dist(x,\ga_i)|)\,\e^{-C_2\dist(x,\ga_i)/2} &\text{if }n=3
  \end{cases}
\]
for the function $v_i$ after arguing as in Eq.~\eqref{vi}.

The symmetry of the domain $\cU$ implies that its Green's function is invariant under this symmetry group: $G_\cU(x,y)=G_\cU(\tau_t^\cA(x),\tau_t^\cA(y))$. Therefore, the above estimate implies that the function
\begin{equation*}
v(x):=\sum_{i=1}^N\sum_{t\in\ZZ^\ell} v_i\big(\tau^\cA_t(x)\big)
\end{equation*}
is well defined, for the above sum converges uniformly on compact subsets of $\overline{\cU}\minus\ga$, with
\[
\ga:=\bigcup_{i=1}^N\bigcup_{t\in\ZZ^\ell}\tau^\cA_t(\ga_i)
\]
being the union of all suitably translated copies of the half-lines $\ga_i$. 
The function $v$, which is of class $C^\infty(\overline{\cU}\minus\ga)\cap L^1\loc(\cU)$, therefore satisfies the equation
\[
(k^2-\De)v=\sum_{i=1}^N\sum_{t\in\ZZ^\ell} (\tau^\cA_t)_*\,\mu_i\,,
\]
in $\cU$ and the boundary condition $v|_{L_0}=0$; indeed, we shall next show that $v$ is the local solution we wanted to construct. In the above equation, $(\tau^\cA_t)_*$ denotes the push-forward. From its definition it is clear that $v$ inherits the symmetries of the domain, meaning that
\begin{equation}\label{vtau}
v(\tau^\cA_t(x))=v(x)\,.
\end{equation}

The next step of the proof consists in showing that the local solution $v$ satisfies the hypotheses of the $C^1$ noncompact stability theorem, that is, there are a half-neighborhood $V\subset\cU$ of the component $L_0$ with $\|v\|_{C^2(V)}<\infty$ and some positive constant $\eta$ such that $\inf_{V}|\nabla v|>0$
and any component of $v^{-1}(c)$ connected with $V$ is contained in $V$ for all $0<c<\eta$. By the symmetry of $v$ (cf.\ Eq.~\eqref{vtau}), it suffices to prove the claim in $\cU\cap Q^\cA$. Since $\cU\cap Q^\cA$ equals $\Om\minus K$, a close look at the proof of Theorem~\ref{T.finite} shows that the proof we gave for tentacled hypersurfaces carries over almost verbatim to the present situation. Indeed, the difference between $D^\al v(x)$ and the function $D^\al \tv_i(x)$ introduced in the relation~\eqref{tildevi} can be computed as in Eq.~\eqref{roll}, the only changes being that one must use the Green's function $G_\cU(x,y)$ instead of $G_\Om(x,y)$ and
the sum over the set of integers $\{j:1\leq j\leq N \text{ and } j\neq i\}$ is to be replaced by a sum over
\[
\big\{ (j,t):1\leq j\leq N,\; t\in\ZZ^\ell\; \text{and}\; j\neq i\text{ if }t=0\big\}\,.
\]
As Lemmas~\ref{L.GS} and~\ref{L.GOmS} and their proofs obviously remain valid with the domain $\Om$ replaced by $\cU$, on account of Lemma~\ref{L.periodic} the rest of the argument remains unchanged and the above claim follows. This allows us to apply Corollary~\ref{C.Thom} as in Theorem~\ref{T.finite} to show that the level set $v^{-1}(c_0)$ is diffeomorphic to the component $L_0$ for any sufficiently small positive constant $c_0$.

Let us now treat the case when the intersection $\cU\cap Q^\cA$ is relatively compact. Then we take a point $y_0$ in this intersection and let
\[
v_0(x):=G_\cU(x,y_0)
\]
be the Green's function of the domain $\cU$ with pole at $y_0$. 
By Lemma~\ref{L.periodic} and Theorem~\ref{T.Jap} one has the exponential bound
\[
v_0(x)\leq C_1|x-y_0|^{2-n}\e^{-C_2|x-y_0|}\,,
\]
which allows us to show that the function
\[
v(x):=\sum_{t\in\ZZ^\ell} v_0\big(\tau^\cA_t(x)\big)
\]
is a well defined solution of the equation
\[
(k^2-\De)v=\sum_{t\in\ZZ^\ell} \de_{\tau^\cA_t(y_0)}
\]
with the symmetry property~\eqref{vtau}. Therefore to show that the function $v$ is the local solution we are looking for it is enough to prove the gradient bound $|\nabla v|\geq C>0$ and the boundedness of the second-order derivatives of $v$ in a saturated neighborhood of the set $L_0\cap Q^\cA$, which is straightforward by Hopf's boundary point lemma and the compactness of $L_0\cap \overline{Q^\cA}$. The rest of the argument goes as in the previous  case, where $\cU\cap Q^\cA=\Om\minus K$.

Hence, applying the same reasoning to all the connected components of the hypersurface $L$ we can use the same argument as in Theorem~\ref{T.finite} to infer that there is a diffeomorphism $\Psi$ of $\RR^n$ with $\|\Psi-\id\|_{C^1(\RR^n)}<\ep/2$, a neighborhood $U$ of $\Psi(L)$ and a local solution $w$ of the equation $(\De-k^2)w=0$ in the closure of $U$ which satisfy the properties (i)--(iv) in the proof of Theorem~\ref{T.finite}. The theorem then follows from the Approximation and Noncompact Stability Theorems~\ref{T.approx} and~\ref{T.stab} as in the proof of Theorem~\ref{T.finite}.
\end{proof}

\begin{remark}\label{R.monsters}
  It is clear that the result and the method of proof remain valid also if $L$ is allowed to have an infinite number of connected components provided we impose appropriate `uniform' assumptions. For example, it suffices to impose that the components of $L$ do not accumulate and that each component be isometric to an element of a fixed, finite collection $\{L_1,\dots, L_r\}$ of tentacled hypersurfaces (possibly disconnected and of infinite type). Likewise, the method of proof also works when the domain $\cU$ is not necessarily periodic but there is a periodic tentacled domain $\hat\cU$ and a compact subset $K$ of $\RR^n$ such that $\cU\minus K=\hat\cU\minus K$.
\end{remark}

\section{Joint level sets}
\label{S.compact}

In this section we shall study joint level sets of solutions to elliptic equations, which amounts to considering transverse intersections of level sets. In doing so, we will find it convenient to consider separately the cases where the components of the joint level set are all compact or not. This difference in treatment is due to the fact that when all the components are compact one can exploit a fairly general better-than-uniform approximation result to deal with a wide class of equations, while when there are some noncompact components the situation is considerably more subtle and our approach relies on our results for single level sets. (It is worth analyzing why the treatment of a single compact level set, for equations that admit these kind of level sets, is considerably less involved than the case of noncompact level sets; for this, we refer to Appendix~\ref{S.appendix}, where this problem is considered in detail.)

\subsection{The compact case}
\label{SS.1}

Here we shall consider the compact joint level sets $u_1^{-1}(c_1)\cap\cdots \cap u_m^{-1}(c_m)$ of solutions to  equations
\begin{equation*}
T_ru_r=0
\end{equation*}
in $\RR^n$. Here each $T_r$ is a linear elliptic differential operator of second order with real analytic coefficients and the number $m$ of functions is at least $2$. As we shall see, the reason why we can consider more general equations in this section is that we can prove a better-than-uniform approximation theorem  for locally finite unions of disjoint compact subsets of $\RR^n$. This theorem provides fine control at infinity, which can be combined with the Cauchy--Kowalewski theorem in each component to derive the desired realization results. This approach does not work when some of the components are noncompact because the approximation result given in Theorem~\ref{T.approx} is not fine enough to deal with the domains of definition of the local Cauchy--Kowalewski solutions, which can be very narrow at infinity.

For the sake of completeness, we start with the following standard lemma, which allows us to approximate a smooth submanifold by an analytic one (we will later on apply this result in the case of codimension $1$ in the preparations to apply the Cauchy--Kowalewski theorem):

\begin{lemma}\label{L.embed}
  Let $L$ be a locally finite union of pairwise disjoint compact codimension-$m$ submanifolds of $\RR^n$  with trivial normal bundle. Then one can transform $L$ by a diffeomorphism $\Psi$ of $\RR^n$, arbitrarily close to the identity in any $C^p$ norm, so that $\Psi(L)$ is an analytic submanifold.
\end{lemma}
\begin{proof}
Let $L_b$ denote the connected components of the submanifold $L$, with
$b$ ranging over an at most countable set $B$. Since $L_b$ has trivial
normal bundle, there is a tubular neighborhood $W_b$ of $L_b$ and a
$C^\infty$ trivializing map $\Theta_b:W_b\to\RR^m$ with
$\Theta_b^{-1}(0)=L_b$. We can safely assume that the closures of the
sets $W_b$ are pairwise disjoint. For any $\de_b>0$, Whitney's
approximation theorem~\cite[Theorem 1.6.5]{Na68} enables us to take a
real analytic submersion $\hTe_b:W_b\to\RR^m$ with
$\|\hTe_b-\Theta_b\|_{C^p(W_b)}$. The hypotheses of
Theorem~\ref{T.stab} being automatically satisfied by the compactness
of the component $L_b$ and the fact that the map $\Theta_b$ is a
trivialization, we can now apply Theorem~\ref{T.stab} and
Remark~\ref{R.Thom} with $(f,g,U,V,L)$ equal to
$(\Theta_b,\hTe_b,W_b,W_b,L_b)$ to derive the existence of
diffeomorphisms $\Psi_b$ of $\RR^n$ such that:
\begin{enumerate}
\item $\hTe_b^{-1}(0)=\Psi_b(L_b)$.
\item $\|\Psi_b-\id\|_{C^p(\RR^n)}<\ep$ and the support of $\Psi_b-\id$ is contained in $W_b$.
\end{enumerate}
Hence these diffeomorphisms naturally define the desired diffeomorphism $\Psi$ by
\[
\Psi(x):=\begin{cases}
  \Psi_{b}(x) &\text{if }x\in W_b\,,\\
  x &\text{if }x\not\in\bigcup_{b\in B}W_b\,.
\end{cases}
\]
\end{proof}

Instead of Theorem~\ref{T.approx}, in this section we will use the
following better-than-uniform approximation result, which is modeled
upon a theorem of Bagby and Gauthier~\cite{BG88}. This result, which
is valid for general analytic elliptic operators but requires an
essential compactness assumption, is proved through an iterative
argument using the Lax--Malgrange theorem (which accounts for the
conditions we impose on the set $S$) and a suitably chosen exhaustion
of $\RR^n$:

\begin{lemma}\label{L.LM}
  Suppose that $S\subset\RR^n$ is a locally finite union of compact sets of nonempty interior whose complements do not have any relatively compact components. If $w$ is a local solution of the equation $T_rw=0$ in the set $S$, then it can be approximated in $S$ by a global solution of this equation in the $C^p$ better-than-uniform sense. That is, for any $p$ and any positive continuous function $\ep(x)$ there is a function $v$ satisfying the equation $T_rv=0$ in $\RR^n$ such that, pointwise in $S$,
  \[
\max_{|\al|\leq p}\big|D^\al v(x)-D^\al w(x)\big|<\ep(x)\,.
\]
\end{lemma}
\begin{proof}
To control the $C^p$ norm of the difference $v-w$ as required by the
function $\ep(x)$, we will introduce some positive constants $\ep_j$
associated to this function and to an exhaustion of $\RR^n$ that we
shall define next. For this, let us denote by $S_b$ the connected
components of the set $S$, where $b$ ranges over an at most countable
set $B$. As $S$ is the locally finite union of $S_b$, we can take an
exhaustion $\emptyset=:K_0 \subset K_1\subset K_2\subset\cdots$  of
$\RR^n$ by compact sets such that:
\begin{enumerate}
\item The union of the interiors $\stackrel{\circ}K_j$ of the sets $K_j$ is the whole $\RR^n$.
\item For each $j$, the complements of $K_j$ and of $S\cup K_j$ are connected.
  \item If some $S_b$ intersects $K_j$, then $S_b$ is contained in the interior of $K_{j+1}$.
\end{enumerate}
Now we can take any positive numbers $\ep_j$ such that
  \begin{equation}\label{epn}
\ep_j<\frac16\min_{x\in K_{j+1}}\ep(x)\quad\text{and}\quad \sum_{k=j+1}^\infty\ep_k<\ep_j
\end{equation}
for all $j\geq1$. We also set $\ep_0:=0$.

We now proceed by induction. We make the induction hypothesis that there are functions $v_j$ satisfying $T_rv_j=0$ in $\RR^n$ and such that, for any $s\geq1$,
\begin{subequations}\label{induction}
  \begin{align}
    \bigg\| w-\sum_{j=1}^sv_j\bigg\|_{C^p(S\cap (K_{s+1}\minus K_s))}&<\ep_s\,,\label{ind1}\\
  \bigg\| w-\sum_{j=1}^sv_j\bigg\|_{C^p(S\cap (K_{s}\minus K_{s-1}))}&<\ep_s+2\ep_{s-1}\,,\label{ind2}\\
    \big\|  v_s\big\|_{C^p(K_{s-1})}&<\ep_s+\ep_{s-1}\,.\label{ind3}
  \end{align}
\end{subequations}

Let us begin by proving the induction hypotheses for $s=1$. As the
complement of $S\cap K_2$ does not have any relatively compact
components, the Lax--Malgrange theorem~\cite[Theorem 3.10.7]{Na68}
yields a function $v_1$ which satisfies the equation $T_rv_1=0$ in $\RR^n$ and such that
\[
\big\|w-v_1\big\|_{C^p(S\cap K_2)}<\ep_1\,.
\]
By the definition of the set $K_0$ and of the constant $\ep_0$, it is therefore evident that the induction hypotheses hold in this case.

Let us now assume that the induction hypotheses hold for all $1\leq s\leq k$ and use this assumption to prove them for $s=k+1$. For this purpose, let us define a function $w_k$ on $S\cup K_k$ by setting $w_k|_{K_k}:=0$ and
\[
w_k|_{S_b}:=\begin{cases}
  w-\sum\limits_{j=1}^kv_j &\text{if }S_b \text{ intersects }K_{k+2}\minus\stackrel{\circ}K_{k+1},\\
  0 &\text{if }S_b \text{ does not  intersect }K_{k+2}\minus\stackrel{\circ}K_{k+1}.
\end{cases}
\]
The definition of the exhaustion and the hypothesis~\eqref{ind1} guarantee that $T_rw_k=0$ and
\begin{equation}\label{fm}
  \big\| w_k\big\|_{C^p(K_k\cup(S\cap K_{k+1}))} \leq  \bigg\| w-\sum_{j=1}^kv_j\bigg\|_{C^p(S\cap( K_{k+1}\minus K_k))}<\ep_k\,.
\end{equation}
Since the complement of the set $K_{k+2}\cap (S\cup K_k)$ does not
have any bounded components by the conditions we imposed on the exhaustion, a further application of the Lax--Malgrange theorem allows us to take a solution $v_{k+1}$ of $T_rv_{k+1}=0$ in $\RR^n$ such that
\begin{equation}\label{gm1}
   \big\|w_k-v_{k+1}\big\|_{C^p(K_{k+2}\cap (S\cup K_k))}<\ep_{k+1}\,.
\end{equation}

Eq.~\eqref{gm1} and the definition of $w_k$ ensure that the hypothesis~\eqref{ind1} also holds for $s=k+1$. Moreover, Eqs.~\eqref{ind1}, \eqref{fm} and~\eqref{gm1} imply the following pointwise $C^p$ bound in the set $S\cap (K_{k+1}\minus K_k)$, valid for all $|\al|\leq p$:
\begin{align*}
  \bigg|D^\al\bigg( w-\sum_{j=1}^{k+1}v_j\bigg)\bigg|&\leq \bigg|D^\al\bigg( w-\sum_{j=1}^{k}v_j\bigg)\bigg| + \big|D^\al v_{k+1}\big|\\
  &<\ep_k+ \big| D^\al(w_k-v_{k+1})\big|+ \big|D^\al w_k\big|<\ep_{k+1}+2\ep_k\,,
\end{align*}
This proves the second induction hypothesis~\eqref{ind2} for $s=k+1$. Moreover,
\[
\big\| v_{k+1}\big\|_{C^p(K_k)}\leq \big\|w_k-v_{k+1}\big\|_{C^p(K_k)}+ \big\| w_k\big\|_{C^p(K_k)}<\ep_{k+1}+\ep_k
\]
as a consequence of Eqs.~\eqref{fm} and~\eqref{gm1}, so the remaining induction hypothesis~\eqref{ind3} also holds for $s=k+1$. The induction argument is then complete.

Let us now define the global solution $v$ as
\[
v:=\sum_{j=1}^\infty v_j\,,
\]
the sum converging $C^p$-uniformly by the definition of constants $\ep_j$ and the induction hypothesis~\eqref{ind3}. Since each $v_j$ satisfies $T_rv_j=0$, this $C^p$ convergence ensures that the function $v$ satisfies the equation $T_rv=0$ too. Besides, from the conditions~\eqref{epn} we imposed on the constants $\ep_j$ and the induction hypotheses~\eqref{induction} it follows that in each set $S\cap (K_{k+1}\minus K_k)$ one has the pointwise $C^p$ estimate
\begin{align*}
  \big|D^\al(w-v)\big|  &\leq \bigg|D^\al\bigg(w-\sum_{j=1}^{k+1}v_j\bigg)\bigg|+\big|D^\al v_{k+2}\big| +  \bigg|D^\al\bigg(\sum_{j=k+3}^\infty v_j\bigg)\bigg|\\
  &<(\ep_{k+1}+2\ep_k) +(\ep_{k+2}+\ep_{k+1}) +\sum_{j=k+3}^\infty(\ep_j+\ep_{j-1})\\
  &<2\ep_k+4\ep_{k+1}<\min_{x\in K_{k+1}}\ep(x)
\end{align*}
for any $k$ and all $|\al|\leq p$, as we wanted to show.
\end{proof}

We now have all the ingredients needed for the proof of Theorem~\ref{T.3}:
%
\begin{proof}[Proof of Theorem~\ref{T.3}]
  Our goal is to obtain the local solution by means of a Cauchy problem. For this, it is convenient to observe that Lemma~\ref{L.embed} ensures that, by perturbing the submanifold a little if necessary, there is no loss of generality in assuming that $L$ is a real analytic submanifold of $\RR^n$.

  Let us denote by $L_b$ the connected components of $L$, with $b$
  taking values in an at most countable set $B$. We start by realizing
  $L_b$ as the intersection of $m$ hypersurfaces $\Si_{br}$. As each
  component $L_b$ also has trivial normal bundle, we can take an
  analytic trivialization $\Theta_b:W_b\to\RR^m$, where $W_b$ is a
  tubular neighborhood of $L_b$ and $\Theta_b^{-1}(0)=L_b$. We denote
  the components of $\Theta_b$ by $(\theta_{b1},\dots,\theta_{bm})$ and consider the analytic hypersurfaces $\Si_{br}:=\theta_{br}^{-1}(0)\subset W_b$, with $1\leq r\leq m$. By the definition of $\Theta_b$, it is apparent that these hypersurfaces intersect transversally at $L_b=\Si_{b1}\cap\cdots\cap \Si_{bm}$.

Now that we have expressed the component $L_b$ as the intersection of $m$ real analytic hypersurfaces $\Si_{br}$, we can consider the following Cauchy problems:
\begin{equation*}
T_rv_{br}=0\,,\qquad v_{br}|_{\Si_{br}}=c_r\,,\qquad \pd_\nu v_{br}|_{\Si_{br}}=1\,.
\end{equation*}
Here $\pd_\nu$ denotes a normal derivative at the hypersurface
$\Si_{br}$. The differential operator $T_r$ being analytic and
elliptic, the Cauchy--Kowalewski theorem grants the existence of a
solution $v_r$ to this Cauchy problem in the closure of a neighborhood
$U_{br}$ of the hypersurface $\Si_{br}$. As the hypersurfaces $\Si_{br}$ intersect transversally at $L_b$ and the gradient of the solution $v_{br}$ coincides with the unit normal of $\Si_{br}$ on this hypersurface, it stems that 
\[
C_b:=\inf_{x\in L_b}\min_{|\om|=1}\big|\om_1\,\nabla v_{b1}(x)+\cdots+ \om_m\,\nabla v_{bm}(x)\big|
\]
is positive. By the compactness of $L_b$, the sets $U_{br}$ can be chosen small enough so that
\begin{equation}\label{Cb}
\inf_{x\in U_b}\min_{|\om|=1}\big|\om_1\,\nabla v_{b1}(x)+\cdots+ \om_m\,\nabla v_{bm}(x)\big|>\frac{C_b}2\,,
\end{equation}
where the set $U_b:=U_{b1}\cap\cdots\cap U_{bm}$ can be assumed to be saturated by the functions $v_{b1},\dots, v_{bm}$ without loss of generality. We can also suppose that the sets $U_b$ are pairwise disjoint.

Let us set $U:=\bigcup_{b\in B}U_b$ and define the local solutions $v_r$ of the equation $T_rv_r=0$ on the set $U$ as $v_r|_{U_{b}}:=v_{br}$. By Lemma~\ref{L.LM}, this local solution can be approximated in the $C^p$ better-than-uniform sense by global solutions $u_r$ of the equations $T_ru_r=0$. That is, for any given positive continuous function $\ep(x)$ in $\overline U$ (to be specified later) we can assume that
\[
\max_{|\al|\leq p}\big|D^\al u_r(x)-D^\al v_r(x)\big|<\ep(x)
\]
in the closure of the set $U$. In view of Eq.~\eqref{Cb} and the compactness of $\overline{U_b}$, Theorem~\ref{T.stab} and Remark~\ref{R.Thom} guarantee the existence of a constant $\ep_b>0$ such that if $\ep(x)<\ep_b$ in $U_b$ there is a diffeomorphism $\Phi_b$ of $\RR^n$, arbitrarily close to the identity in the $C^p$ norm and equal to the identity outside $U_b$, which transforms the submanifold $L_b$ into a connected component of the joint level set $u_1^{-1}(c_1)\cap\cdots \cap u_m^{-1}(c_m)$. As the supports of each map $\Phi_b-\id$ are pairwise disjoint and we can choose $\ep(x)$ so that it is smaller than $\ep_b$ in each set $U_b$, the theorem then follows by letting the diffeomorphism $\Phi$ be equal to $\Phi_b$ in each set $U_b$ and equal to the identity in the complement of $U=\bigcup_{b\in B} U_b$.
\end{proof}

\begin{example}
It is known that any exotic $m$-sphere smoothly embeds in $\RR^{2m}$
with trivial normal bundle~\cite{HLS65} and that any locally finite
link in $\RR^3$ has trivial normal bundle too, so Theorem~\ref{T.3} obviously furnishes a positive answer to Questions~\ref{Q.2} and~\ref{Q.3}. In particular, there are two harmonic functions $u_1,u_2$ in $\RR^3$ such that $u_1^{-1}(0)\cap u_2^{-1}(0)$ contains a knot in each isotopy class.
\end{example}

\begin{remark}
When $L$ is a finite union of compact codimension-$m$ components and $T_r=\De$ for all $r$, one can actually prove that $u_1,\dots, u_m$ can be chosen to be harmonic polynomials by proceeding as above and using an approximation theorem of Paramonov~\cite{Pa94} instead of the Lax--Malgrange theorem.
\end{remark}

\subsection{The noncompact case}
\label{SS.2}

In this subsection we present a realization theorem for noncompact
joint level sets $u_1^{-1}(c_1)\cap\cdots \cap u_m^{-1}(c_m)$ of solutions $u_r$ of the equations $(\De-k_r^2)u_r=0$, for any real constants $k_r$. It is assumed that the number of functions $m$ is at least $2$. As in the case of a single level set, it will be crucial to control the geometry of the submanifolds at infinity using asymptotic translation symmetries. In order to implement this idea, we will resort to the codimension-$m$ analog of a tentacled hypersurface, which can be defined in terms of tentacled hypersurfaces in a very simple manner:

\begin{definition}\label{D.joint}
  A (periodic) {\em codimension-$m$ tentacled submanifold} is the
  transverse intersection of $m$ tentacled hypersurfaces  (possibly
  disconnected and of infinite type).
\end{definition}

Before going on, let us briefly discuss why this is the suitable
codimension-$m$ analog of a tentacled hypersurface. From
Definition~\ref{D.tentacled}, it is clear that the properties one
would require for a submanifold $L$ ( for concreteness, connected and of finite type) to be a codimension-$m$ analog of a tentacled hypersurface are the following:
\begin{enumerate}
\item $L$ must have trivial normal bundle, as this is a topological obstruction to being a transverse intersection of level sets.
\item There must be $J$ embedded images $\Pi_j$ of $\RR^{n-1}$ in $\RR^n$, which divide $\RR^n$ into two domains $H_j^+$ and $H_j^-$ satisfying certain properties analogous to the conditions~(i)--(iv) in Definition~\ref{D.tentacled}. In particular, each component of the intersection of $L$ with $\Pi_1\cup \cdots\cup \Pi_J$ is compact, has dimension $n-m-1$ and is contained in an affine $(n-m)$-plane. If its components are called $\La_i$, the submanifold $L$ has a finite number of ends and are all isometric to some product $\La_i×(0,\infty)$.
\end{enumerate}
That the codimension-$m$ submanifold $L$ is then the transverse intersection of tentacled hypersurfaces $L^1,\dots, L^m$ readily follows from the fact that, as a consequence of these properties, $L$ has a tubular neighborhood whose boundary is a tentacled hypersurface. A very similar argument is also valid in the periodic, disconnected case.

\begin{example}
  In particular, from Example~\ref{E.algebraic} it follows that the
  transverse intersection of $m$ nonsingular algebraic open hypersurfaces is diffeomorphic to a codimension-$m$ tentacled submanifold. For instance, the torus of genus $g$ with $N$ ends can be realized as a $2$-dimensional tentacled submanifold in any $\RR^n$.
\end{example}

We shall next present the proof of Theorem~\ref{T.2}, which asserts that if $L$ is a tentacled submanifold of codimension $m$ and $\tL$ is a finite union of compact submanifolds of the same codimension with trivial normal bundle, then there are $m$ solutions $u_r$ to the equations $(\De-k_r^2)u_r=0$ in $\RR^n$ having a joint level set diffeomorphic to $L\cup\tL$ (provided, of course, that $L$ and $\tL$ are disjoint). It is apparent that the `compact' part of the argument can be easily tackled arguing as in the proof of Theorem~\ref{T.3}, so most of the proof will be devoted to the case of tentacled submanifolds:

\begin{proof}[Proof of Theorem~\ref{T.2}]
  Let $L_0$ be a connected component of the codimension-$m$
  submanifold $L\cup\tL$. We want to prove that there are $m$ local
  solutions $v_r$ of the equations $(\De-k_r^2)v_r=0$, defined in a
  suitable domain whose closure contains $L_0$, which have a joint
  level set diffeomorphic to $L_0$ and satisfy the hypotheses of the
  $C^1$ Stability Theorem~\ref{T.stab}. When $L_0$ is compact, this
  immediately follows from Theorem~\ref{T.3}, so we will henceforth
  assume that $L_0$ is noncompact. By hypothesis, $L_0$ is then given
  by the transverse intersection of $m$ (connected but possibly
  periodic) tentacled hypersurfaces $L^1,\dots,L^m$.
  
  Let us concentrate for the moment in the tentacled hypersurface
  $L^r$. As we showed in the proof of Theorem~\ref{T.finite}
  (or~\ref{T.infinite}, if $L^r$ is periodic), there is a
  local solution $v_r$, defined in a half-neighborhood
  $\overline{V_r}$ of $L^r$ and satisfying the conditions (i)--(iv) in the proof of Theorem~\ref{T.finite}, namely:
  \begin{enumerate}
    \item There is some positive and arbitrarily small constant whose corresponding level set  $v_r^{-1}(c_r)$ is diffeomorphic to the hypersurface $L^r$. This diffeomorphism can be chosen supported in a small neighborhood of $\overline{V_r}$ and close to the identity in the $C^1$ norm.
    \item The neighborhood $V_r$ is saturated by the function $v_r$.
    \item The gradient of the function $v_r$ satisfies the lower bound $|\nabla v_r|\geq C>0$ in the set $V_r$ and on $ L_0$ can be written as $\nabla v_r=-|\nabla v_r|\,\nu_r$, where $\nu_r$ is the outer normal at the hypersurface. Besides, the $C^2$ norm $\|v_r\|_{C^2(V_r)}$ is finite.
    \item The complement of the set $V_r$ does not have any bounded components.
\end{enumerate}

We claim that the joint level set $v_1^{-1}(c_1)\cap \cdots\cap v_m^{-1}(c_m)$ is diffeomorphic to the component $L_0$ via a $C^1$-small diffeomorphism. To prove this, we start by noticing that the set $V:=V_1\cap\cdots \cap V_m$ is saturated by the map $(v_1,\dots, v_m)$, which moreover has bounded second-order derivatives in $V$ by the condition~(iii) above. To show that the gradient condition of the stability theorem holds, we aim to prove that for any unit vector $\om=(\om_1,\dots,\om_m)$ the function
\[
F_\om(x):=\big|\om_1\nabla v_1(x)+\cdots +\om_m\nabla v_m(x)\big|^2
\]
satisfies the bound
\[
    F_\om(x)\geq C_1
\]
in the set $V$  for some positive constant $C_1$ independent of $\om$. As the hypersurfaces $L^1,\dots, L^m$ intersect transversally at the component $L_0$ and their geometry is controlled at infinity because they are tentacled, an easy geometric argument shows that the quantity
\[
C_2:=\inf_{x\in L_0}\min_{|\om|=1}\big|\om_1\,\nu_1(x)+\cdots+ \om_m\,\nu_m(x)\big|^2
\]
is positive. As $\nabla v_r=-|\nabla v_r|\,\nu_r$ on $L_0$ with $|\nabla v_r(x)|\geq C$ by the condition~(iii) above, this immediately implies that, for any $x\in L_0$,
\[
  F_\om(x)=\big|\om_1|\nabla v_1(x)|\,\nu_1(x)+\cdots+ \om_m|\nabla v_m(x)|\,\nu_m(x)\big|^2\geq C^2C_1
\]
is bounded from below by some positive constant independent of $\om$. As the $C^2$ norm of the functions $v_r$ is bounded and their gradients satisfy $|\nabla v_r|\geq C$ in $V$, it is straightforward to check  that, by taking the saturated set $V$ smaller if necessary, we can assume that $\inf_{|\om|=1}\inf_{x\in V}F_\om(x)$ is also positive.

As a consequence of the above digression, we can apply
Corollary~\ref{C.Thom} with
$(f,g,L)=\big((v_r)_{r=1}^m,(v_r-c_r)_{r=1}^m,L_0)$ (the domain $V$ in
Corollary~\ref{C.Thom} coincides with the current set $V$, while $U$
is a suitably chosen neighborhood of $\overline V$). This way we infer that, for any sufficiently small positive constants $c_r$, the joint level set $v_1^{-1}(c_1)\cap\cdots \cap v_m^{-1}(c_m)$ is diffeomorphic to the component $L_0$ via a $C^1$-small diffeomorphism supported in a neighborhood of $\overline V$, as we claimed.

Applying the same argument to each connected component of the
submanifold $L$ we obtain local solutions $w_r$ of the equations
$(\De-k_r^2)w_r=0$ that have a joint level set diffeomorphic to $L$
via a $C^1$-small diffeomorphism of $\RR^n$ and satisfy the hypotheses of the $C^1$ stability theorem. The number of connected components of $L$ being finite, this allows us to apply the Approximation Theorem~\ref{T.approx} and the Stability Theorem~\ref{T.stab} to complete the proof of the theorem (the details follow just as in the proof of Theorem~\ref{T.finite} and are thus omitted).
\end{proof}

\begin{remark}
  Obviously, the analogs of Remarks~\ref{R.solomonic} and~\ref{R.monsters} also apply to the case of joint level sets. 
\end{remark}

\appendix

\section{The case of a compact level set}
\label{S.appendix}

For the sake of completeness, in this Appendix we shall consider the
construction of solutions to the equation $(\De-q)u=0$ in $\RR^n$
having a level set diffeomorphic to a union of compact hypersurfaces,
provided the function $q$ satisfies a mild positivity assumption. Of
course, the result cannot hold true for harmonic functions as a
consequence of the maximum principle. 

A compact hypersurface separates $\RR^n$ into an inner (bounded) domain and an outer (unbounded) domain. In the following theorem we shall assume that the inner domains associated to each component of the hypersurface are pairwise disjoint. The reason for this is that a key element of the proof is the better-than-uniform approximation result established in Lemma~\ref{L.LM}, which makes use of this hypothesis. Upon comparison with the proof of Theorem~\ref{T.1}, this illustrates how the proof drastically simplifies in the case of compact components, as many of the subtleties that appeared in the former case are now absent.

\begin{theorem}\label{T.appendix}
  Let $L$ be a locally finite union of compact hypersurfaces of
  $\RR^n$ and let $q$ be a nonnegative, real analytic function on $\RR^n$ that is not identically zero. Suppose that the inner domains corresponding to the connected components of $L$ are pairwise disjoint and fix a positive constant $c$. Then one can transform the hypersurface $L$ by a diffeomorphism $\Phi$ of $\RR^n$, arbitrarily close to the identity in any $C^p$ norm, so that $\Phi(L)$ is a union of connected components of the level set $u^{-1}(c)$ of a function that satisfies the equation $(\De-q)u=0$ in $\RR^n$.
\end{theorem}
\begin{proof}
  Let us denote the connected components of $L$ by $L_b$, the index $b$ ranging over an at most countable set $B$, and let $\Om_b$ stand for the inner domain of the component $L_b$. We saw in Lemma~\ref{L.embed} that, perturbing $L$ a little if necessary via a $C^p$-small diffeomorphism, there is no loss of generality in assuming that all the components $L_b$ are real analytic.

  A natural way of constructing a local solution of the equation having a level set diffeomorphic to $L_b$ is via a boundary value problem in the domain $\Om_b$.  It is standard that the boundary value problem
  \[
(\De-q)v_b=0\quad \text{in }\Om_b\,,\qquad v_b|_{L_b}=c
  \]
  has a unique solution, which is given by
  \[
v_b(x):=c-c\int_{\Om_b} G_{\Om_b}(x,y)\,q(y)\,dy\,.
  \]
  Here $G_{\Om_b}(x,y)$ is the Green's function of the operator
  $\De-q$ in the domain $\Om_b$, defined as in Proposition~\ref{P.Green}. The function $v_b$ satisfies the equation $(\De-q)v_b=0$ in the closure of $\Om_b$ by the fact that the function $q$ and the domain $\Om_b$ are analytic~\cite{Mo58}, is smaller than $c$ in $\Om_b$ by the maximum principle, and its gradient $\nabla v_b$ is nonzero on $L_b$ by Hopf's boundary point lemma~\cite{GT98}. Hence we can now define a function $v$ on the closed set $S:=\bigcup_{i\in I}\overline{\Om_b}$ that satisfies $(\De-q)v=0$ by setting $v|_{\overline\Om_b}:=v_b$. Obviously $v^{-1}(c)=L$.

  We shall next prove that one can use Lemma~\ref{L.LM} on better-than-uniform approximation to find a global solution $u$ of the equation such that $L$ is diffeomorphic to a union of connected components of $u^{-1}(c)$. This makes use of Remark~\ref{R.Thom} on the topological stability theorem, from which it stems that, given any component $L_b$ and any positive constant $\ep_0$, there is a positive constant $\de_b$ such that for any function $u$ with $\|u-v\|_{C^p(\Om_b)}<\de_b$ one can find a diffeomorphism $\Phi_b$ of $\RR^n$ with $\|\Phi_b-\id\|_{C^p(\RR^n)}<\ep_0$ mapping $L_b$ onto a connected component of $u^{-1}(c)$. It should be noticed that the saturation and gradient conditions of the stability theorem are obviously satisfied as a consequence of the gradient estimate $\nabla v_b\neq0$ on $L_b$ and the compactness of each component. One can also assume that $\Phi_b-\id$ is supported in a neighborhood $U_b$ of $\BOm_b$, with the sets $U_b$ pairwise disjoint (cf.\ Theorem~\ref{T.stab}).

  To complete the proof of the theorem, let us take a positive
  continuous function~$\ep(x)$ which is smaller than $\de_b$ in each
  set $\BOm_b$. By Lemma~\ref{L.LM} there is a function $u$ satisfying
  the equation $(\De-q)u=0$ in $\RR^n$ and the following $C^p$
  better-than-uniform bound in the set $S$:
  \[
\max_{|\al|\leq p}\big|D^\al u(x)-D^\al v(x)\big|<\ep(x)\,.
\]
By construction, $\Phi(L)$ is then a union of connected components of $u^{-1}(c)$, where the diffeomorphism $\Phi$ is defined as $\Phi(x):=\Phi_b(x)$ if $x$ belongs to some set $U_b$ and $\Phi(x):=x$ otherwise. 
\end{proof}
\begin{remark}
It is clear that the same proof works for a wider class of positive
second-order operators, but we will not pursue this issue here.  The
case of the equation $(\De+k^2)u=0$, considered in Question~\ref{Q.2}
in the context of joint level sets, can also be easily dealt with. To construct the local solution having a zero set diffeomorphic to $L$, it suffices to consider the first Dirichlet eigenfunction $v_b$ of each domain $\Om_b$, which satisfies $(\De+\la_b)v_b=0$ for some constant $\la_b>0$, and to apply suitable dilations and rigid motions to the functions $v_b$. With this starting point, one can apply the ideas of the proof of Theorem~\ref{T.appendix} to show the existence of a solution of the equation $(\De+k^2)u=0$ in $\RR^n$ such that $\Phi(L)$ is a union of connected components of $u^{-1}(0)$, but it should be noticed that in this case the diffeomorphism $\Phi$ is not granted to be close to the identity.
\end{remark}
\begin{remark}
Theorem~\ref{T.appendix} can be readily combined with
Theorem~\ref{T.1}  to deal with compact and noncompact components for
the equation $(\De-k^2)u=0$ at the same time, when $k\neq0$.
\end{remark}

To illustrate Theorem~\ref{T.appendix}, let us consider the following easy application:

\begin{example}
  There is a function satisfying $\De u=u$ in $\RR^3$ whose zero set
  contains all the compact surfaces, that is,  has a component of
  genus $g$ for all nonnegative integers $g$ (cf.\ Figure~\ref{F.dim2}). This is an obvious consequence of Theorem~\ref{T.appendix}, as there is a countable number of compact surfaces modulo diffeomorphism (indeed, they are customary labeled by their genus).
\end{example}

\begin{figure}[t]
  \centering
  \psfrag{d}{\dots}
  \includegraphics[scale=0.62,angle=0]{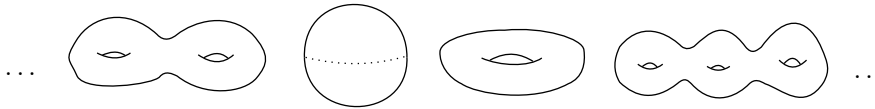}
\caption{A solution of the equation $\De u=u$ in $\RR^3$ whose zero set contains a compact surface of each genus.} 
\label{F.dim2}
\end{figure}

\section*{Acknowledgments}

The authors are indebted to  Yoichi Miyazaki for pointing out
Reference~\cite{Ta97} and for explaining to them how to deal with
estimates for higher derivatives of the Green's function, and to Paul M.\ Gauthier and Juan J.L.\ Velázquez for valuable comments. The authors are also grateful to Lawrence Conlon for his permission to reproduce his infinite jungle gym in our Figure 1. This work is supported in part by the MICINN under grants no.\ FIS2011-22566 (A.E.) and MTM2010-21186-C02-01 (D.P.-S.) and by Banco Santander--UCM under grant no.\ GR58/08-910556 (A.E.). The authors acknowledge the Spanish Ministry of Science and Innovation for financial support through the Ramón y Cajal program.

\end{document}